\documentclass[11pt, noadjust, reqno]{amsart}
\usepackage[dvipsnames]{xcolor}
\usepackage[utf8]{inputenc}
\usepackage{amssymb,amsmath,amsthm,amsfonts,fullpage,float,latexsym,bbm,microtype,cite,fancyvrb}
\usepackage{tikz, tikz-cd}
\usepackage{enumerate}
\usetikzlibrary{decorations.pathreplacing}
\usepackage{hyperref}
\hypersetup{colorlinks=true, linkcolor=blue, citecolor=magenta, filecolor=magenta, urlcolor=magenta}
\usepackage{bm}
\allowdisplaybreaks[1]

\newtheorem*{theorem*}{Theorem}

\newtheorem{theorem}{Theorem}[section]
\newtheorem{lemma}[theorem]{Lemma}
\newtheorem{proposition}[theorem]{Proposition}
\newtheorem{corollary}[theorem]{Corollary}
\newtheorem*{corollary*}{Corollary}
\newtheorem{conjecture}[theorem]{Conjecture}
\newtheorem*{conjecture*}{Conjecture}

\newtheorem*{zconjecture*}{Zabrocki's Conjecture}

\newtheorem{introthm}{Theorem}

\newtheorem{introcor}[introthm]{Corollary}

\theoremstyle{definition}
\newtheorem{definition}[theorem]{Definition}
\newtheorem{example}[theorem]{Example}

\theoremstyle{remark}
\newtheorem{remark}[theorem]{Remark}

\numberwithin{equation}{section}

\makeatletter
\newtheorem*{rep@theorem}{\rep@title}\newcommand{\newreptheorem}[2]{%
\newenvironment{rep#1}[1]{%
\def\rep@title{\bf #2 \ref{##1}}%
\begin{rep@theorem}}%
{\end{rep@theorem}}}
\makeatother
\newreptheorem{theorem}{Theorem}
\newreptheorem{corollary}{Corollary}

\renewcommand{\emptyset}{\varnothing}

\newcommand{\Q}{\mathbb{Q}}

\newcommand{\Z}{\mathbb{Z}}
\newcommand{\C}{\mathbb{C}}

\renewcommand{\S}{\mathfrak{S}}
\newcommand{\cC}{\mathcal{C}}
\newcommand{\cY}{\mathcal{Y}}

\DeclareMathOperator{\Hilb}{Hilb}
\DeclareMathOperator{\Frob}{Frob}
\DeclareMathOperator{\Char}{Char}
\DeclareMathOperator{\Sym}{Sym}

\DeclareMathOperator{\sgn}{sgn}
\DeclareMathOperator{\area}{area}

\DeclareMathOperator{\bounce}{bounce}
\DeclareMathOperator{\dinv}{dinv}
\DeclareMathOperator{\diag}{diag}
\DeclareMathOperator{\coarm}{\mathsf{coarm}}
\DeclareMathOperator{\coleg}{\mathsf{coleg}}
\DeclareMathOperator{\Hom}{Hom}
\DeclareMathOperator{\Ind}{Ind}
\DeclareMathOperator{\Res}{Res}

\DeclareMathOperator{\Span}{span}

\newcommand{\newword}[1]{\emph{\textbf{#1}}}

\newcommand\schroderpath[7]{
  \begin{tikzpicture}[scale=.5]
    \foreach \x/\y in {#4}{
      \fill[red!40] (\x,\y) rectangle +(1,1);
    }
    \foreach \x/\y in {#5}{
      \fill[red!40] (\x,\y) -- ++(1,0) -- ++(0,1) -- cycle;
    }
    \foreach \x/\y in {#6}{
      \fill[red!40] (\x,\y) -- ++(0,1) -- ++(1,0) -- cycle;
    }
    \draw[help lines] (0,0) grid (#2,#2);
    \draw[dashed] (0,0) -- (#2,#2);
    \coordinate (current) at (0,0);
    \draw[line width=1.5pt] (current) foreach \s in {#1}{
      -- ++({\s==1 ? 0 : 1},{\s==0 ? 0 : 1}) coordinate (current)
    };
    \foreach \x/\y/\txt in {#7}{
      \node at (\x,\y) {\txt};
    }
  \end{tikzpicture}
}

\begin{document}

\title{Sign components of diagonal superspace coinvariants}

\author[N. González]{Nicolle González}
\address{Department of Mathematics\\
         University of British Columbia, Vancouver, BC, Canada}
\email{nicolle@math.ubc.ca}

\author[J. Lentfer]{John Lentfer}
\address{Department of Mathematics\\
         University of California, San Diego, CA, USA}
\email{jlentfer@ucsd.edu}

\author[H. Mularczyk]{Hanna Mularczyk}
\address{Department of Mathematics\\
        Massachusetts Institute of Technology, Cambridge, MA, USA}
\email{hannamul@mit.edu}

\begin{abstract}
We prove the sign-isotypic components of the coinvariant rings $R_n^{(2,1)}$ and $R_n^{(2,0)} \otimes R_n^{(0,1)}$ are isomorphic and show that the triply-graded multiplicity of this sign character is the Schr\"oder polynomial $S_n(q,t,a)$, divided by $1+a$.
This settles the sign-character component of a conjecture of Zabrocki (2019) on a module for the Delta theorem and proves a conjecture of F. Bergeron (2020) on the multiplicity of the sign character of $R_n^{(2,1)}$. 
Finally, using a result of Hogancamp (2017), we enhance a recent result of Gorsky--Mellit (2026) which relates the Khovanov--Rozansky homology of the $(n,n+1)$-torus knot to $R_n^{(2,0)} \otimes R_n^{(0,1)}$, by showing that the associated Poincar\'e series for this knot can be computed from the sign component of $R_n^{(2,1)}$. 

\end{abstract}

\maketitle

\section{Introduction}

Symmetric group coinvariant rings are central figures in algebraic combinatorics, with deep connections to algebraic geometry, representation theory, and symmetric functions. 
The most basic example, the classical coinvariant ring $R_n^{(1,0)}$, of dimension $n!$  \cite{Artin}, is isomorphic to the regular representation of the symmetric group $\mathfrak{S}_n$ \cite{Chevalley} and to the cohomology ring of the flag variety \cite{Borel, Leray}. 
In the 1990's, Garsia and Haiman introduced a bivariate generalization referred to as the \newword{diagonal coinvariant ring} $R_n^{(2,0)}$ \cite{Haiman1994}. 
This ring attracted immense interest from the combinatorics community, in particular concerning its dimension and Frobenius character. 
By realizing $R_n^{(2,0)}$ geometrically via the Hilbert scheme on $\C^2$, Haiman was able to prove the \emph{$(n+1)^{n-1}$ conjecture}.
This established that its Frobenius series is given by the action of the nabla operator $\nabla$ on the $n$th elementary symmetric function $e_n$, and consequently, that its dimension is $(n+1)^{n-1}$ \cite{Haiman2002}.
These results had immediate and resounding impact. 
Haiman's work on Hilbert schemes has since been very influential in algebraic geometry.
The character formula $\nabla e_n$ spurred a search for a purely combinatorial formulation for this symmetric function, culminating in the celebrated \emph{Shuffle conjecture} of Haglund, Haiman, Loehr, Remmel, and Ulyanov \cite{HHLRU2005}, which was proven in remarkable fashion a decade later by Carlsson and Mellit \cite{CarlssonMellit2018}.

More generally, one can define the \newword{$(k,\ell)$-bosonic-fermionic coinvariant ring} $R_n^{(k,\ell)}$ as the quotient of the polynomial ring in $k$ commuting (bosonic) and $\ell$ anticommuting (fermionic) families of variables,
$\C[\bm{x}^{(1)}, \ldots, \bm{x}^{(k)}, \bm{\theta}^{(1)}, \ldots, \bm{\theta}^{(\ell)}]$, by the ideal generated by the positive-degree invariants of the diagonal $\mathfrak{S}_n$-action (see Definition~\ref{def:bosonic-fermionic}). 
At parameters $(1,0)$ and $(2,0)$, this recovers the aforementioned classical and diagonal coinvariant rings. 
At $(1,1)$, we obtain the \newword{superspace coinvariant ring} $R_n^{(1,1)}$, which was originally studied by N. Bergeron, Colmenarejo, Li, Machacek, Sulzgruber, and Zabrocki at the Fields Institute in 2018. 
Subsequently, the multiplicity of its sign character \cite{SwansonWallach1}, its dimension and Hilbert series \cite{RhoadesWilson2023}, a monomial basis \cite{SaganSwanson2024, Angarone2024}, and its Frobenius series \cite{MuraiRhoadesWilson} have all been determined. 

At $(2,1)$, we obtain the \newword{diagonal superspace coinvariant ring} $R_n^{(2,1)}$, introduced by Zabrocki in \cite{Zabrocki2019}.
Far less is known about $R_n^{(2,1)}$ than $R_n^{(2,0)}$: beyond the structural results of \cite{Bergeron2020,Lentfer-Supersymmetry} and a conjectural basis \cite{HaglundSergel}, its Frobenius series and dimension are only conjectural \cite{Zabrocki2019}.
Conjecturally, its Frobenius series is closely related to the acclaimed \emph{Delta theorem} in the same way $R_n^{(2,0)}$ is to the Shuffle theorem. 
Namely, in \cite{Zabrocki2019} Zabrocki conjectured that $R_n^{(2,1)}$ provides a module-theoretic lift of the symmetric function appearing in the Delta theorem \cite{HaglundRemmelWilson2018, DAdderioMellit}:
\begin{zconjecture*}[\cite{Zabrocki2019}]
For $n \geq 1$,
    \begin{equation}\label{eq:delta thm}
        \Frob(R_n^{(2,1)}; q,t; a) = \sum_{d=0}^{n-1} \Delta'_{e_{n-1-d}}(e_n) a^d.
    \end{equation}
\end{zconjecture*}
\noindent In particular, a combinatorial formula for the right hand side of Equation~\eqref{eq:delta thm} is precisely the content of the Delta theorem. 

\subsection{Main construction and results} In this paper we study the sign-isotypic component of the diagonal superspace coinvariant ring $R_n^{(2,1)}$. 
Sign-isotypic components of coinvariant rings have rich combinatorics: $(R_n^{(1,0)})_{\sgn}$ is spanned by the Vandermonde determinant and $(R_n^{(2,0)})_{\sgn}$ has bigraded Hilbert series given by the \emph{$(q,t)$-Catalan polynomial}, so its dimension is the \emph{Catalan number} $C_n$. 
The Catalan number $C_n$ enumerates $n\times n$ Dyck paths, which are lattice paths above the diagonal consisting of $n$ horizontal and $n$ vertical steps. Naturally, this suggests that $(R_n^{(2,1)})_{\sgn}$ should carry interesting combinatorics generalizing that of $(R_n^{(2,0)})_{\sgn}$.
Small examples (see Example \ref{ex:sign-tensor}) suggested that the graded dimension of $(R_n^{(2,1)})_{\sgn}$ coincides with that of $(R_n^{(2,0)} \otimes R_n^{(0,1)})_{\sgn}$, even though $R_n^{(2,1)}$ is a proper quotient module of $R_n^{(2,0)} \otimes R_n^{(0,1)}$.\footnote{For $2 \leq n \leq 6$, we can check that the sign character is the only irreducible character whose graded multiplicities in the two modules agree.}

Our flagship result upgrades this observed agreement of graded dimensions to a grading-preserving isomorphism between the sign-isotypic components of $R_n^{(2,0)} \otimes R_n^{(0,1)}$ and $R_n^{(2,1)}$ (Theorem~\ref{introthm:main}). 
To prove it, we appeal to a well-known identification of the sign components of $R_n^{(2,0)} \otimes R_n^{(0,1)}$ and the hook isotypic components of $R_n^{(2,0)}$. 
We then pass to harmonic space, realizing $R_n^{(2,0)}$ as $H_n^{(2,0)}$, the \newword{space of diagonal harmonics}, which is a space of polynomials in $\C[\bm{x},\bm{y}]$ in the kernel of certain differential operators. This presentation enables us to explicitly identify the hook isotypic components for $\lambda = (d+1,1^{n-d-1})$ inside $H_n^{(2,0)}$, and define a grading-preserving map from this subspace into the fermionic-degree $d$ component of $(R_n^{(2,1)})_{\sgn}$. 
Summing over all values of $d$ yields a grading-preserving injection $(R_n^{(2,0)} \otimes R_n^{(0,1)})_{\sgn} \hookrightarrow(R_n^{(2,1)})_{\sgn}$.  Since the natural surjection $\rho: R_n^{(2,0)} \otimes R_n^{(0,1)} \twoheadrightarrow R_n^{(2,1)}$ restricts to a surjection on sign-isotypic components, we obtain an isomorphism (proven in Theorems~\ref{thm:Phi-injection} and~\ref{thm:main}):
\begin{introthm}\label{introthm:main}
    For $n\geq 1$, the sign-isotypic components of $R_n^{(2,1)}$ and $R_n^{(2,0)} \otimes R_n^{(0,1)}$ are isomorphic as triply-graded $\S_n$-modules.
\end{introthm}
\noindent We remark that although constructing sign-isotypic elements in the harmonic model mirrors the approach for $R_n^{(1,1)}$ by Swanson--Wallach \cite{SwansonWallach1} and for $R_n^{(0,3)}$ by the second author \cite{Lentfer2025}, the arguments used here are new. Moreover, our proof does not utilize Haiman's operator theorem \cite{Haiman1994,Haiman2002}, which is the typical method for constructing $H_n^{(2,0)}$. 
Instead, we construct the necessary elements of $H_n^{(2,0)}$ explicitly.

\subsection{Combinatorial consequences} The isomorphism above has immediate combinatorial consequences. Just like the Catalan numbers $C_n$ count the number of $n \times n$ Dyck paths, the \newword{Schr\"oder numbers} $S_n$ \cite{schroder} enumerate $n \times n$ Schr\"oder paths, that is, lattice paths above the diagonal consisting of vertical, horizontal, and diagonal steps. A bivariate polynomial generalization of $C_n$, known as the $(q,t)$-Catalan polynomial $C_n(q,t)$, was introduced by Garsia and Haiman \cite{Garsia-Haiman} and later generalized to the Schr\"oder setting by Egge, Haglund, Killpatrick, and Kremer \cite{EggeHaglundKillpatrickKremer} with the introduction of the trivariate \newword{Schr\"oder polynomials} $S_n(q,t,a)$.
By considering Schr\"oder paths with no diagonal steps above the highest north step, enumerated by the \newword{little Schr\"oder numbers} $\widetilde{S}_n$, one obtains the  
\newword{little Schr\"oder polynomials} $\widetilde{S}_n(q,t,a) = \frac{1}{1+a}S_n(q,t,a)$. It is a conjecture of F. Bergeron \cite{Bergeron2020} that the dimension of $(R_n^{(2,1)})_{\sgn}$ is precisely $\widetilde{S}_n$.

Now, since by work of Haiman \cite{Haiman2002}, the Frobenius character of $R_n^{(2,0)}$ is $\nabla e_n$, it follows from Haglund's \emph{$(q,t)$-Schr\"oder theorem} \cite{Haglund2004} that $\langle \Frob(R_n^{(2,0)} \otimes R_n^{(0,1)}), e_n \rangle$ equals the little Schr\"oder polynomial $\widetilde{S}_n(q,t,a)$. Thus, as a consequence of Theorem \ref{introthm:main}, we obtain a combinatorial formula for the Hilbert series of $(R_n^{(2,1)})_{\sgn}$ and its corresponding dimension (in Corollaries \ref{cor:sign-character-schroder} and \ref{cor:Bergeron-conj}):
\begin{introthm}\label{introthm:signcharacter}
    The sign-isotypic component of $R_n^{(2,1)}$ has Hilbert series given by
    \begin{equation*}
    \Hilb((R_n^{(2,1)})_{\sgn}; q,t;a
    )=\left\langle \Frob(R_n^{(2,1)};q,t;a), e_n\right\rangle = \widetilde{S}_n(q,t,a).
\end{equation*}
Hence, $(R_n^{(2,1)})_{\sgn}$ has dimension $\widetilde{S}_n$, so F. Bergeron's conjecture \cite{Bergeron2020} is true.
\end{introthm}

Recalling Zabrocki's conjecture, it follows from Zabrocki's 4-variable Catalan theorem \cite{Zabrocki-4Catalan-2016} (see also Proposition~\ref{prop:zabrocki-alternating}) that the sign-character component of the right-hand side of Equation~\eqref{eq:delta thm} also yields $\widetilde{S}_n(q,t,a)$. We verify the sign-character component of Zabrocki's conjecture (as Corollary~\ref{cor:zabrocki-conjecture-sign}):
\begin{introcor} 
  For all $n\geq 1$,
    \begin{equation*}
       \left\langle \Frob(R_n^{(2,1)};q,t;a),e_n \right\rangle = \left\langle \sum_{d=0}^{n-1} \Delta'_{e_{n-1-d}}(e_n) a^d, e_n \right\rangle.
    \end{equation*}
\end{introcor}

\subsection{Connections with link homology} Our results relate to link homology. 
\newword{Khovanov--Rozansky homology} is a triply-graded link homology theory developed in \cite{KR1, KR2, Kh} which categorifies the HOMFLY-PT polynomial and associates to a knot (or link) a triply-graded chain complex of Soergel bimodules. 
Computing the homology and the graded dimensions of these complexes for large families of knots is a notoriously difficult problem and extremely active research area (see for instance \cite{aganagic, cherednik, EH16, CD16, CD17, hog-mellit, GMO, KT24, Mellit-Homology}). 
In particular, due to a series of groundbreaking conjectures by Gorsky, Negu\c{t}, Oblomkov, Rasmussen, Shende, and others \cite{Gorsky_Oblomkov_Rasmussen_Shende_2014, GorskyNegut, Gorsky-Catalan, GNR21, ORS, OS}, particular interest has been directed towards finding representation theoretic and combinatorial interpretations of the Khovanov--Rozansky homology of algebraic links \cite{CGHM, GL1, GL2, Hog17, GMV1, GMV2, Mellit}. 

The most fundamental algebraic links are the \newword{torus knots} $T(n,m)$, where $n,m$ are relatively prime positive integers. 
Hogancamp \cite{Hog17} proved a conjecture of Gorsky \cite{Gorsky-Catalan} by computing the Poincar\'e series $\mathcal{P}_K(q,t,a)$ of the Khovanov--Rozansky homology of $K= T(n,n+1)$ and showing it equals $\frac{1+a}{1-q}\widetilde{S}_n(q,t,a)$. 
In light of Theorem \ref{introthm:signcharacter}, the next result is immediate (proved as Theorem~\ref{thm:torusknot}):
\begin{introthm} The Hilbert series of  $(R_n^{(2,1)})_{\sgn}$ determines the Poincar\'e series of the Khovanov--Rozansky homology of the $(n,n+1)$-torus knot $K$. Namely, 
  \begin{equation*}\label{eq:intro-KR-(2,1)}
  \mathcal{P}_K(q,t,a) = \frac{1+a}{1-q} \; 
\Hilb((R_n^{(2,1)})_{\sgn}; q,t;a).
\end{equation*}
\end{introthm}
Our result is the second to establish a direct link between the Khovanov--Rozansky homology of the $(n,n+1)$-torus knot and bosonic-fermionic coinvariant rings. 
In early 2026, Gorsky and Mellit \cite{GorskyMellit} constructed an isomorphism between the homology of $T(n,n+1)$ and the sign-isotypic component of the tensor product $R_n^{(2,0)} \otimes R_n^{(0,1)}$. 
Although their work combined with our Theorem~\ref{thm:main} implies Theorem~\ref{thm:torusknot}, our derivation and methodology are distinct and independent. 
In particular, Gorsky and Mellit's proof fundamentally relies on Haiman's operator theorem \cite{Haiman1994, Haiman2002} to construct $H_n^{(2,0)}$, which, as noted above, we do not use.

\subsection{Further directions} The Poincar\'e series for $(n,m)$-torus links was computed by Hogancamp and Mellit \cite{hog-mellit}, and recently extended to Coxeter knots by Caprau, the first author, Hogancamp, and Mazin \cite{CGHM}. 
In all these cases, combinatorial formulas in terms of generalized \emph{triangular Schr\"oder polynomials} were given. 
On the other hand, although from computational evidence certain geometric subspaces of $R_n^{(2,1)}$ seem related to the $(n,bn+1)$-torus knot, these spaces do not yet have a harmonic-space interpretation, impeding generalizations of the present work. 

The homology of $T(n,m)$ can be computed from certain representations of rational Cherednik algebras.
While Gordon \cite{Gordon} connected these representations when $m=n+1$ (and thus the homology of $T(n,n+1)$) to $R_n^{(2,0)}$, there is no known or expected connection for generic $m$.
In particular, outside of $m=n+1$, these modules have no known description as the quotient of a polynomial ring, which is the starting point for a characterization as a coinvariant space.
Thus a `super' generalization of this theory is unknown.
Nonetheless, we hope the results in this article initiate further exploration relating link homology theories to some `higher' analogue of superspace diagonal coinvariant spaces.

\subsection{Outline of the paper}
The paper is organized as follows. In Section~\ref{sec:Schroder-polynomials} we overview the necessary combinatorial background, including both  Schr\"oder polynomials $S_n(q,t,a)$ and $\widetilde{S}_n(q,t,a)$ and their properties. In Section~\ref{sec:background}, we recall the various bosonic-fermionic coinvariant rings, compute sign-characters for certain tensor products of these spaces, and clarify the relationship between $R_n^{(2,1)}$ and $R_n^{(2,0)}\otimes R_n^{(0,1)}$. The core of the paper lies in Section~\ref{sec:main}. In Section~\ref{subsec:identifyhook} we pass to harmonic spaces and identify the required hook-isotypic components inside $H_n^{(2,0)}$ in Proposition~\ref{prop:hook-hom-iso} and Theorem~\ref{thm:tildehook-hom-iso}. Then by way of Theorem~\ref{thm:harmonic-2-1}, in Section~\ref{subsec:harmoniclift} we construct the map lifting the hook components of $H_n^{(2,0)}$ into the sign component of $H_n^{(2,1)}$. In Section~\ref{subsec:main theorem} we establish the main results, Theorems \ref{thm:Phi-injection} and \ref{thm:main}, where the isomorphism between sign-isotypic components is proven and the Hilbert series is shown to equal the little Schr\"oder polynomial. In Section~\ref{subsec:zabrocki-conj} we continue with Corollary~\ref{cor:zabrocki-conjecture-sign} proving Zabrocki's conjecture for the sign-isotypic component. The paper concludes in Section~\ref{sec:link-homology} where we prove Theorem~\ref{thm:torusknot} and establish the direct connection between $T(n,n+1)$ and $(R_n^{(2,1)})_{\sgn}$.

\section*{Acknowledgements}

We thank Sylvie Corteel, Eugene Gorsky, Mark Haiman, Xinchun Ma, Arun Ram, Brendon Rhoades, Christopher Ryba, George Seelinger, and José Simental Rodríguez for helpful conversations.
N.G. was supported by the Natural Sciences and Engineering Research Council of Canada, RGPIN-2026-07206. 
J.L. was partially supported by the National Science Foundation Graduate Research Fellowship DGE-2146752. H.M. was supported by a National Science Foundation Graduate Research Fellowship under Grant No.\ 2141064.
We thank the mathematical research institute MATRIX in Australia where part of this research was performed while N.G. and J.L. were in attendance at the program ``Symmetric Functions and Stochastic Processes.''

\section{Schr\"oder polynomials}\label{sec:Schroder-polynomials}

A \newword{Schr\"oder path} of order $n$ is a lattice path from $(0,0)$ to $(n,n)$ composed of north steps $N = (0,1)$, east steps $E = (1,0)$, and diagonal steps $D = (1,1)$, which does not go below the diagonal line $y=x$.
For $0 \leq d \leq n$, let $L_{n,n,d}^+$ denote the set of Schr\"oder paths of order $n$ with exactly $d$ diagonal steps. 
Write $L_{n,n}^+ = \bigsqcup_{d=0}^{n} L_{n,n,d}^+$ for the set of all Schr\"oder paths of order $n$.
For more details on Schr\"oder numbers and diagonal coinvariants, see \cite[Chapter 4]{Haglund2008}. 
The cardinality of $L_{n,n}^+$ is the \newword{Schr\"oder number}, which is given by
\begin{equation*}
    S_n = \left|L_{n,n}^+\right| = \frac{2}{n} \sum_{i=0}^{n-1} \binom{n}{i} \binom{n}{i+1} 2^i.
\end{equation*}

Two important statistics on Schr\"oder paths, \newword{area} and \newword{bounce}\footnote{The bounce statistic, and thus the choice of little Schr\"oder paths in what follows, follow a different convention in \cite{EggeHaglundKillpatrickKremer}. We follow \cite{Haglund2008} in this paper.}, are defined in \cite[Chapter 4]{Haglund2008}.
These statistics extend the definitions of area and bounce on Dyck paths (which are Schr\"oder paths with no diagonal steps).
Given a Schr\"oder path $\pi \in L_{n,n}^+$, write $\area(\pi)$ for its area statistic, $\bounce(\pi)$ for its bounce statistic, and $\diag(\pi)$ for the number of diagonal steps.
Define
\begin{equation*}
    S_{n,d}(q,t) = \sum_{\pi \in L_{n,n,d}^+} q^{\area(\pi)} t^{\bounce(\pi)}.
\end{equation*}

We call a Schr\"oder path $\pi$ a \newword{little Schr\"oder path} if it has no diagonal step $D$ above the highest north step $N$. 
Exactly half of the Schr\"oder paths are little Schr\"oder paths.\footnote{There exist other combinatorial criteria by which the Schr\"oder paths can be partitioned into two equinumerous sets, for example, paths with no diagonal step on the main diagonal $y=x$ (see for example \cite{ParkKim}).} 
Let $\tilde{L}_{n,n,d}^+\subseteq L_{n,n,d}^+$ denote the set of little Schr\"oder paths with $d$ diagonal steps.

By convention, a little Schr\"oder path must have at least one north step, so there are none with only diagonal steps; hence we restrict to considering $0 \leq d \leq n-1$ diagonal steps. 
Define $\tilde{L}_{n,n}^+ = \bigsqcup_{d=0}^{n-1} \tilde{L}_{n,n,d}^+$.
The cardinality of $\tilde{L}_{n,n}^+$ is the \newword{little Schr\"oder number}, which is given by
\begin{equation*}
    \widetilde{S}_n = \left|\tilde{L}_{n,n}^+\right| = \frac{1}{n} \sum_{i=0}^{n-1} \binom{n}{i} \binom{n}{i+1} 2^i,
\end{equation*}
and is exactly half of the Schr\"oder number $S_n$.
For combinatorial interpretations of $S_n$ and $\widetilde{S}_n$, see \cite[Chapter 6 Exercise 39]{StanleyEC2}. 
For history of these numbers, see \cite{Stanley-Schroder} and \cite[Chapter 6 Notes]{StanleyEC2}.

Then define
\begin{equation*}
    \widetilde{S}_{n,d}(q,t) = \sum_{\pi \in \tilde{L}_{n,n,d}^+} q^{\area(\pi)}t^{\bounce(\pi)}.
\end{equation*}
We adopt the convention that $\widetilde{S}_{n,d}(q,t) = 0$ whenever $d <0$ or $d \geq n$ to streamline formulas.

\begin{figure}
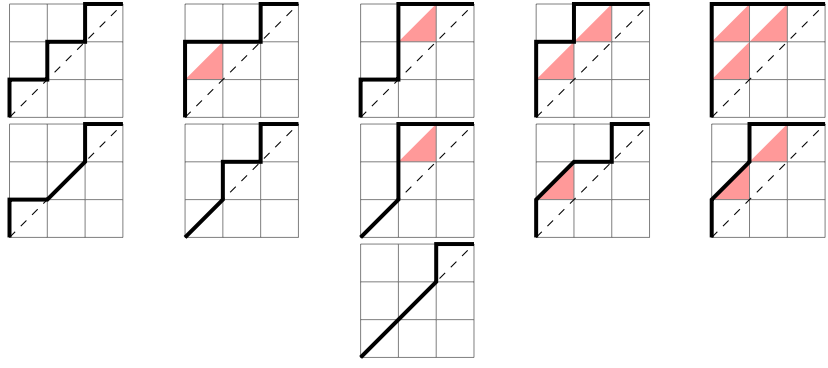

    \centering
    \schroderpath{1,0,1,0,1,0}{3}{0.8}{}{}{}{} \quad 
    \schroderpath{1,1,0,0,1,0}{3}{0.8}{}{0/1}{}{} \quad 
    \schroderpath{1,0,1,1,0,0}{3}{0.8}{}{1/2}{}{} \quad 
    \schroderpath{1,1,0,1,0,0}{3}{0.8}{}{0/1,1/2}{}{} \quad
    \schroderpath{1,1,1,0,0,0}{3}{0.8}{}{0/1,0/2,1/2}{}{} \quad

    \schroderpath{1,0,2,1,0}{3}{0.8}{}{}{}{} \quad 
    \schroderpath{2,1,0,1,0}{3}{0.8}{}{}{}{} \quad 
    \schroderpath{2,1,1,0,0}{3}{0.8}{}{1/2}{}{} \quad 
    \schroderpath{1,2,0,1,0}{3}{0.8}{}{0/1}{}{} \quad 
    \schroderpath{1,2,1,0,0}{3}{0.8}{}{0/1,1/2}{}{} \quad 

    \schroderpath{2,2,1,0}{3}{0.8}{}{}{}{} \quad

    \caption{The eleven little Schr\"oder paths at $n=3$. The number of red triangles is the area statistic.}
    \label{fig:little-Schroder-3}
\end{figure}

\subsection{The \texorpdfstring{$(q,t)$}{(q,t)}-Schr\"oder theorem}
Both polynomials $S_{n,d}(q,t)$ and $\widetilde{S}_{n,d}(q,t)$ can be computed using the nabla operator (defined in Equation~\eqref{eq:nabla-defintion}), connecting them with the theory of diagonal coinvariants; this is the route via which they enter our story.

Let $e_n$ and $h_n$ denote the $n$th elementary and homogeneous symmetric functions and $s_\lambda$ the Schur function for partition $\lambda$ in an infinite number of commuting variables. The \emph{$(q,t)$-Schr\"oder theorem} relates the Schr\"oder polynomials to the so-called hook components of certain symmetric functions by pairing them, via the Hall inner product, with the function $e_{n-d}h_d$. Its classical form below was conjectured by Egge, Haglund, Killpatrick, and Kremer \cite{EggeHaglundKillpatrickKremer} and proven by Haglund \cite{Haglund2004}, and recently generalized by Caprau, the first author, Hogancamp, and Mazin \cite{CGHM}.

\begin{theorem}[\cite{Haglund2004}]\label{thm:qt-Schroeder-two-hooks}
    For all $0 \leq d \leq n$,
    \begin{equation*}
        S_{n,d}(q,t) = \left\langle \nabla e_n, e_{n-d} h_d \right\rangle.
    \end{equation*}
\end{theorem}

Going from all Schr\"oder paths to just the little Schr\"oder paths isolates a single Schur function indexed by a hook shape. 
Note that the hook shape $(d+1,1^{n-d-1})$ is indeed a valid partition for $0 \leq d \leq n-1$.
Then the $(q,t)$-Schr\"oder theorem can be restated as follows.
\begin{theorem}[{\cite[Theorem 4.3]{Haglund2008}}]\label{thm:qt-Schroeder-one-hook}
    For all $0 \leq d \leq n-1$, 
    \begin{equation*}
        \widetilde{S}_{n,d}(q,t) = \left\langle \nabla e_n, s_{(d+1,1^{n-d-1})} \right\rangle.
    \end{equation*}
\end{theorem}

\subsection{Further identities}
Recall the symmetric function identity 
\begin{equation}\label{eq:e-s-identity}
    e_{n-d}h_d = s_{(d+1,1^{n-d-1})} + s_{(d,1^{n-d})} \text{ for all } 0\leq d \leq n,
\end{equation}
where at the boundary cases, we use the convention that a Schur function indexed by something that is not a valid partition is $0$.
We easily deduce the following.
\begin{proposition}\label{prop:Schroder-splitting}
For all $0 \leq d \leq n$,
   \begin{equation} S_{n,d}(q,t) = \widetilde{S}_{n,d}(q,t)+\widetilde{S}_{n,d-1}(q,t). 
    \end{equation}
\end{proposition}

\begin{proof}
    By Theorem~\ref{thm:qt-Schroeder-two-hooks}, Equation~\eqref{eq:e-s-identity}, and Theorem~\ref{thm:qt-Schroeder-one-hook}, we write
    \begin{align*}
        S_{n,d}(q,t) &= \Big\langle \nabla e_n, e_{n-d}h_d \Big\rangle\\
        &=  \left\langle \nabla e_n, s_{(d+1,1^{n-d-1})} \right\rangle +\left\langle \nabla e_n, s_{(d,1^{n-d})} \right\rangle\\
        &= \widetilde{S}_{n,d}(q,t)+\widetilde{S}_{n,d-1}(q,t). 
    \end{align*}
\end{proof}

Inverting this recursion writes each $\widetilde{S}_{n,d}(q,t)$ as an alternating sum of the $S_{n,d}(q,t)$.
\begin{proposition}\label{prop:Schroder-alternating}
For all $0 \leq d \leq n-1$,
    \begin{equation*}
        \widetilde{S}_{n,d}(q,t)
        = \sum_{i=0}^{d}(-1)^{i+d} S_{n,i}(q,t).
    \end{equation*}
\end{proposition}

\begin{proof}
By Proposition~\ref{prop:Schroder-splitting}, $S_{n,i}(q,t) = \widetilde{S}_{n,i}(q,t)+\widetilde{S}_{n,i-1}(q,t)$, so we write
\begin{align*}
    \sum_{i=0}^d (-1)^i S_{n,i}(q,t)
    &= \sum_{i=0}^d (-1)^i \widetilde{S}_{n,i}(q,t) + \sum_{i=0}^d (-1)^i \widetilde{S}_{n,i-1}(q,t)\\
    &= \sum_{i=0}^d (-1)^i \widetilde{S}_{n,i}(q,t) + \sum_{i=0}^{d-1} (-1)^{i+1} \widetilde{S}_{n,i}(q,t)\\
    &= (-1)^d \widetilde{S}_{n,d}(q,t).
\end{align*}
Multiplying by $(-1)^d$ completes the proof.
\end{proof}

\subsection{The Schr\"oder polynomials}
Summing the polynomials $S_{n,d}(q,t)$ or $\widetilde{S}_{n,d}(q,t)$  over all $d$, with $a^d$ recording the grading by the number of diagonal steps, allows us to define the Schr\"oder polynomials.
\begin{definition}[{\cite[Chapter 4]{Haglund2008}}]
    The \newword{Schr\"oder polynomial} $S_n(q,t,a)$ and \newword{little Schr\"oder polynomial} $\widetilde{S}_n(q,t,a)$ are defined by
    \begin{equation}\label{eq:Schroder-polynomial-def}
    S_n(q,t,a) = \sum_{d=0}^n S_{n,d}(q,t) a^d \quad \text{ and } \quad
    \widetilde{S}_n(q,t,a) = \sum_{d=0}^{n-1} \widetilde{S}_{n,d}(q,t) a^d.
    \end{equation}
Equivalently, by Theorems~\ref{thm:qt-Schroeder-two-hooks} and~\ref{thm:qt-Schroeder-one-hook},
    \begin{equation}\label{eq:Schroder-polynomial-def-second}
    S_n(q,t,a) = \sum_{d=0}^n \Big\langle \nabla e_n, e_{n-d} h_d \Big\rangle a^d \quad \text{ and } \quad
    \widetilde{S}_n(q,t,a) =
    \sum_{d=0}^{n-1} \left\langle \nabla e_n, s_{(d+1,1^{n-d-1})}  \right\rangle a^d.
    \end{equation}
\end{definition}

Comparing the two definitions yields a clean factorization.
\begin{proposition}\label{prop:BigSmall-Schroder}
For $n \geq 1$,
    \begin{equation*}
S_n(q,t,a) = (1+a)\widetilde{S}_{n}(q,t,a).
    \end{equation*}
\end{proposition}

\begin{proof}
By equations~\eqref{eq:Schroder-polynomial-def-second} and~\eqref{eq:e-s-identity}, we write
    \begin{align*}
S_n(q,t,a) &= \sum_{d=0}^n \Big\langle \nabla e_n, e_{n-d} h_d \Big\rangle a^d
= \sum_{d=0}^{n-1} \left\langle \nabla e_n, s_{(d+1,1^{n-d-1})} \right\rangle a^d +
\sum_{d=1}^n \left\langle \nabla e_n, s_{(d,1^{n-d})}\right\rangle a^d
\\
&= \sum_{d=0}^{n-1} \left\langle \nabla e_n, s_{(d+1,1^{n-d-1})} \right\rangle a^d +
a \sum_{d=0}^{n-1} \left\langle \nabla e_n, s_{(d+1,1^{n-d-1})}\right\rangle a^{d}
\\
&= (1+a) \widetilde{S}_n(q,t,a).
    \end{align*}
\end{proof}

Directly from the definitions (the $d=0$ terms), we have that $S_n(q,t,0)=\widetilde{S}_n(q,t,0)=C_n(q,t)$ is the $(q,t)$-Catalan polynomial. 

\subsection{Path models and the zeta map}
Collecting the area and bounce statistics over all $d$ gives a single path-based combinatorial formula for the Schr\"oder polynomials $S_n(q,t,a)$ and $\widetilde{S}_n(q,t,a)$:
\begin{align*}
    S_n(q,t,a)
    = \sum_{d=0}^{n}a^d \sum_{\pi \in L_{n,n,d}^+} q^{\area(\pi)} t^{\bounce(\pi)} = \sum_{\pi \in L_{n,n}^+}q^{\area(\pi)} t^{\bounce(\pi)}a^{\diag(\pi)},
\end{align*}
and
\begin{align*}
    \widetilde{S}_n(q,t,a)
    = \sum_{d=0}^{n-1}a^d \sum_{\pi \in \tilde{L}_{n,n,d}^+} q^{\area(\pi)} t^{\bounce(\pi)} = \sum_{\pi \in \tilde{L}_{n,n}^+}q^{\area(\pi)} t^{\bounce(\pi)}a^{\diag(\pi)}.
\end{align*}
It follows immediately that $S_n(1,1,1)= S_n$ and $\widetilde{S}_n(1,1,1)= \widetilde{S}_n$.

There is a zeta map $\zeta$ on Schr\"oder paths, which is a bijection satisfying $\dinv(\pi) = \area(\zeta(\pi))$ and $\area(\pi) = \bounce(\zeta(\pi))$ (see \cite[Chapter 4]{Haglund2008} for the definition of the $\zeta$ map and the $\dinv$ statistic). 
It also preserves the number of diagonal steps.
Applying the zeta map gives another expression for $\widetilde{S}_n(q,t,a)$ in terms of the dinv and area statistics:
\begin{align*}
    \widetilde{S}_n(q,t,a)
        &= \sum_{\pi \in \tilde{L}_{n,n}^+}q^{\area(\pi)} t^{\bounce(\pi)}a^{\diag(\pi)}\\
            &= \sum_{\pi \in {\tilde{L}_{n,n}^+} }q^{\dinv(\zeta^{-1}(\pi))} t^{\area(\zeta^{-1}(\pi))}a^{\diag(\zeta^{-1}(\pi))}\\
    &= \sum_{\tau \in {\zeta^{-1}(\tilde{L}_{n,n}^+} )}q^{\dinv(\tau)} t^{\area(\tau)}a^{\diag(\tau)}.
\end{align*}
Note that $\zeta$ generally does not fix $\tilde{L}_{n,n,d}^+$ as a subset of ${L}_{n,n,d}^+$.

\section{Diagonal coinvariant rings}\label{sec:background}

For a fixed $n \geq 1$, we will denote a family of $n$ variables $\{ z_1,\ldots,z_n \}$ by $\bm{z}$:
\[
\bm{z} := \{ z_1,\ldots,z_n \}.
\]
Consider a family of bosonic (commutative) alphabets $\bm{x}^{(1)}, \dots, \bm{x}^{(k)}$ and fermionic (anticommutative) alphabets $\bm{\theta}^{(1)}, \dots, \bm{\theta}^{(\ell)}$
so that for any $\bullet,\star$ we have:
\begin{equation*}
\begin{aligned}
    &[x^{(\bullet)}_i,x^{(\star)}_j]=[x^{(\bullet)}_i,\theta^{(\star)}_j] =0, \\
    &\theta^{(\bullet)}_i\theta^{(\star)}_j + \theta^{(\star)}_j\theta^{(\bullet)}_i =0,  \qquad \left(\theta^{(\bullet)}_i\right)^2 =0.
\end{aligned}
\end{equation*}
Then the vector space
\[
\C[\bm{x}^{(1)}, \ldots, \bm{x}^{(k)}, \bm{\theta}^{(1)}, \ldots, \bm{\theta}^{(\ell)}] = \C[\bm{x}^{(1)}, \ldots, \bm{x}^{(k)}] \otimes \Lambda[\bm{\theta}^{(1)}, \ldots, \bm{\theta}^{(\ell)}]  \cong
   \Sym\left((\C^n)^{\oplus k}\right) \otimes
   {\textstyle\bigwedge}\left((\C^n)^{\oplus \ell}\right) 
\]
is an $\mathfrak{S}_n$-module via the natural \newword{diagonal action} obtained by permuting each family $\bm{x}^{(\bullet)}$ and $\bm{\theta}^{(\star)}$ simultaneously, where for all $i, \star, \bullet$, and $\tau \in \mathfrak{S}_n$: 
\[
\tau(x^{(\bullet)}_i) = x^{(\bullet)}_{\tau(i)} \quad \text{and} \quad \tau(\theta^{(\bullet)}_i) = 
\theta^{(\bullet)}_{\tau(i)}.
\]

\begin{definition}\label{def:bosonic-fermionic}
The \newword{$(k,\ell)$-bosonic-fermionic coinvariant ring} is defined as the quotient
\begin{equation}\label{eq:coinv} R_n^{(k,\ell)} := \C[\bm{x}^{(1)}, \ldots, \bm{x}^{(k)}, \bm{\theta}^{(1)}, \ldots, \bm{\theta}^{(\ell)}]/I_n^{(k,\ell)},\end{equation}
where the defining ideal
\[ I_n^{(k,\ell)}:=\left\langle \C[\bm{x}^{(1)}, \ldots, \bm{x}^{(k)}, \bm{\theta}^{(1)}, \ldots, \bm{\theta}^{(\ell)}]^{\mathfrak{S}_n}_+\right\rangle\]
is generated by the $\mathfrak{S}_n$-invariant polynomials without constant term under the diagonal action of $\mathfrak{S}_n$.
\end{definition}
In particular, the spaces $R_n^{(k,\ell)}$ inherit the natural polynomial grading and are themselves $(k+\ell)$-graded $\mathfrak{S}_n$-modules.
Consequently, $R_n^{(k,\ell)}$ decomposes as a direct sum of multihomogeneous components, which are themselves $\mathfrak{S}_n$-modules:
\begin{equation*}
R_n^{(k,\ell)} = \bigoplus_{\substack{r_1, \dots, r_k \geq 0\\ s_1,\dots,s_\ell \geq 0}} \left(R_n^{(k,\ell)}\right)_{r_1,\dots,r_k, s_1,\dots,s_\ell}
\end{equation*}
where for any $j$ and each $m$, we set
\[
\mathsf{deg}(x^{(m)}_j) = q_m \qquad \text{and} \qquad \mathsf{deg}(\theta^{(m)}_j) = a_m.
\]

\begin{definition}
The \newword{multigraded Hilbert series} is the sum of the graded dimensions, given by 
\begin{equation*}
\begin{aligned} 
\Hilb(R_n^{(k,\ell)}; q_1,\dots, q_k; a_1,\dots, a_\ell):= \sum_{\substack{r_1, \dots, r_k \geq 0\\ s_1,\dots,s_\ell \geq 0}}\dim \left(R_n^{(k,\ell)}\right)_{r_1,\dots,r_k, s_1,\dots,s_\ell}q_1^{r_1}\dots q_k^{r_k}a_1^{s_1}\dots a_\ell^{s_\ell},
\end{aligned}
\end{equation*}
and the \newword{multigraded Frobenius series} by 
\begin{equation*}
\begin{aligned} 
\Frob(R_n^{(k,\ell)}; q_1,\dots, q_k; a_1,\dots, a_\ell):= \sum_{\substack{r_1, \dots, r_k \geq 0\\ s_1,\dots,s_\ell \geq 0}} F\Char\left(R_n^{(k,\ell)}\right)_{r_1,\dots,r_k, s_1,\dots,s_\ell}q_1^{r_1}\dots q_k^{r_k}a_1^{s_1}\dots a_\ell^{s_\ell},
\end{aligned}
\end{equation*}
where $F$ denotes the Frobenius characteristic map and $\Char$ denotes the character.
\end{definition}

The Hilbert series can be recovered from the Frobenius series by pairing it with $e_1^n$:
\[
\Hilb(R_n^{(k,\ell)}; q_1,\dots, q_k; a_1,\dots, a_\ell)=\left\langle \Frob(R_n^{(k,\ell)}; q_1,\dots, q_k; a_1,\dots, a_\ell), e_1^n \right\rangle .\]
Similarly, the Hilbert series of the sign-isotypic component of $R_n^{(k,\ell)}$ can be recovered by pairing $\Frob(R_n^{(k,\ell)})$ with $e_n$ under the Hall inner product.

In this article we are particularly interested in two special cases of $R_n^{(k,\ell)}$, when $k=2$ and $\ell \in \{0,1\}$. We also utilize results on the classical and fermionic coinvariant rings $R_n^{(1,0)}$ and $R_n^{(0,1)}$. In these cases, we consistently denote by $\bm{x}$ and $\bm{y}$ the two bosonic alphabets and by $\bm{\theta}$ the single fermionic alphabet. 
Then we use the parameters 
\[q=q_1, \quad t=q_2, \quad \text{ and } \quad a=a_1\]
to denote the $\bm{x},\bm{y}$, and $\bm{\theta}$ degrees, respectively. 

\begin{definition} 
The \newword{diagonal coinvariant ring} $R_n^{(2,0)}$ and \newword{diagonal superspace coinvariant ring} $R_n^{(2,1)}$ are defined as the spaces
    \begin{align*} 
R_n^{(2,0)} := \C[\bm{x}, \bm{y}]/I_n^{(2,0)} \qquad \text{ and } \qquad
R_n^{(2,1)} := \C[\bm{x}, \bm{y},\bm{\theta}]/I_n^{(2,1)}.
\end{align*}
\end{definition}
In particular, $R_n^{(2,1)}$ admits the following description in terms of symmetric and exterior algebras:
\begin{equation*}
    R_n^{(2,1)} = \left(\left(\Sym\C^n\right)^{\otimes 2} \otimes \left(\bigwedge\C^n\right)\right)\Big/\left\langle \left(\left(\Sym\C^n\right)^{\otimes 2} \otimes \left(\bigwedge\C^n\right)\right)^{\mathfrak{S}_n}_+ \right\rangle.
\end{equation*}

\subsection{Characters for diagonal coinvariant spaces} 

The \newword{modified Macdonald polynomials} $\tilde{H}_\lambda$ originally arose in the work of Garsia and Haiman in their study of diagonal harmonics as a plethystic modification of the integral form Macdonald polynomials \cite{Macdonald, GarsiaHaiman1993, GarsiaHaiman1996}. 
In addition to arising as the Frobenius characters of the Garsia--Haiman modules, these polynomials form a Schur-positive basis for the ring of symmetric functions with coefficients in $\C(q,t)$ (see \cite{HaimanCDM}). 

For a partition $\mu$ (in French notation, so that parts are left-justified with the largest part on the bottom) and box $c \in \mu$ let:
\begin{align*}
\coarm_\mu(c)&:= \text{ the number of boxes strictly to the left of $c$ in $\mu$,}
\\
\coleg_\mu(c)&:= \text{ the number of boxes strictly below $c$ in $\mu$}.
\end{align*}
Define the following polynomials in $\Q[q,t]$: 
\begin{align*}
    B_\mu := \sum_{c \in \mu} q^{\coarm_\mu(c)} t^{\coleg_\mu(c)} \qquad \text{and} \qquad
    T_\mu := \prod_{c \in \mu} q^{\coarm_\mu(c)} t^{\coleg_\mu(c)},
\end{align*}
where the sum and product are taken over all boxes $c$ in the Young diagram of the partition $\mu$.

The \newword{nabla operator} $\nabla$ \cite{BergeronGarsia} is defined to be the following eigenoperator on the Macdonald basis: 
\begin{equation}\label{eq:nabla-defintion}
    \nabla \tilde{H}_\mu = T_\mu \tilde{H}_\mu \text{ for all $\mu$}.
\end{equation}

More generally, for any symmetric function $f$, the \newword{delta operators} $\Delta_f$ and $\Delta_f'$ \cite{HaglundRemmelWilson2018, BergeronGarsiaHaimanTesler} are the eigenoperators on the Macdonald basis defined by 
\begin{equation*}
    \Delta_f \tilde{H}_\mu = f[B_\mu]\tilde{H}_\mu \text{ and } \Delta_f' \tilde{H}_\mu = f[B_\mu-1]\tilde{H}_\mu \text{ for all $\mu$}.
\end{equation*}
Here, $f[\cdot]$ denotes plethystic substitution (see \cite{HaimanCDM}).
In particular, on symmetric functions that are homogeneous of degree $n$, the nabla operator is a special case of the delta operator: 
\[
\nabla = \Delta_{e_n}.
\]

In the groundbreaking paper \cite{Haiman2002}, Haiman used the geometry of the Hilbert scheme on $\C^2$ to compute the bigraded Frobenius series of $R_n^{(2,0)}$ in terms of the nabla operator.

\begin{theorem}[\cite{Haiman2002}]\label{thm:haiman}
    For $n \geq 1$,
    \begin{equation*}
        \Frob(R_n^{(2,0)};q,t) =  \nabla e_n.
    \end{equation*}
\end{theorem}

Although a similar identity for the diagonal superspace coinvariant ring is not known, Zabrocki made the following conjecture, which has attracted interest due to the role of $\Delta'_{e_{n-1-d}}(e_n)$ in the Delta theorem (conjectured in \cite{HaglundRemmelWilson2018}, and proven in \cite{DAdderioMellit} and later in \cite{BHMPS-Delta}).
\begin{conjecture}[\cite{Zabrocki2019}]\label{conj:zabrocki}
For $n \geq 1$,
    \begin{equation*}
        \Frob(R_n^{(2,1)}; q,t; a) = \sum_{d=0}^{n-1} a^d\Delta'_{e_{n-1-d}}(e_n).
    \end{equation*}
\end{conjecture}
In this paper, we will prove the sign-character component of Zabrocki's conjecture (see Section~\ref{subsec:zabrocki-conj}).

The ring $R_n^{(2,1)}$ may be realized as a $U(\mathfrak{gl}(2|1)) \otimes \C[\mathfrak{S}_n]$-module \cite{Lentfer-Supersymmetry}.
Certain irreducible $\mathfrak{gl}(2|1)$-characters are \emph{super Schur polynomials} $s_\lambda(q,t/a)$ (see \cite{BereleRegev} and \cite[Appendix A.2.2]{ChengWangBook}). 
Hence $\Frob(R_n^{(2,1)}; q,t; a)$ may be written as a positive integral sum of products of super Schur polynomials with Schur polynomials \cite{Lentfer-Supersymmetry}. 
In particular, by extracting the sign-isotypic component, we conclude that $\left \langle\Frob(R_n^{(2,1)}; q,t; a), s_{(1^n)}\right\rangle$ is a positive integral sum of super Schur polynomials $s_\lambda(q,t/a)$.

\subsection{The tensor product \texorpdfstring{$R_n^{(2,0)} \otimes R_n^{(0,1)}$}{Rn(2,0) otimes Rn(0,1)}} 
In this subsection we relate $R_n^{(2,1)}$ to $R_n^{(2,0)} \otimes R_n^{(0,1)}$. The sign-isotypic component of $R_n^{(2,0)} \otimes R_n^{(0,1)}$ is computed easily using Kronecker products in Corollary~\ref{cor:(2,0)(0,1)bound}. We also show there is a natural surjection $\rho: R_n^{(2,0)} \otimes R_n^{(0,1)} \twoheadrightarrow R_n^{(2,1)}$ in Lemma~\ref{lem:quotient}.
\medskip

Recall that 
\begin{equation*}
R_n^{(0,1)}= \left( \bigwedge\C^n\right)\Big/\left\langle \left( \bigwedge\C^n \right)^{\mathfrak{S}_n}_+ \right\rangle \cong \bigoplus_{i=0}^{n-1} \left(\bigwedge\nolimits^{i} S^{(n-1,1)} \right)
\cong \bigoplus_{i=0}^{n-1} S^{(n-i,1^i)},
\end{equation*}
where $S^\lambda$ denotes the Specht module indexed by partition $\lambda$, $S^{(n-1,1)}$ and $\epsilon=S^{(1^n)}$ are the reflection and sign representations of $\mathfrak{S}_n$, respectively, and 
\[\epsilon \otimes S^{(n-i,1^i)} = \epsilon \otimes \bigwedge\nolimits^{i} S^{(n-1,1)} \cong \bigwedge\nolimits^{n-1-i} S^{(n-1,1)} \cong 
S^{(i+1,1^{n-1-i})}.\]
Consequently, the Frobenius series of $R_n^{(0,1)}$ is a sum of Schur functions indexed by hook shapes (see {\cite[Lemma 4.10]{HaglundSergel}}):
\begin{equation*} \Frob(R_n^{(0,1)};a) =  \sum_{k=0}^{n-1} a^k s_{(n-k,1^k)}.\end{equation*}

As noted in \cite[Section 3.3]{GorskyMellit}, it follows from Schur's lemma that for any $\mathfrak{S}_n$-module $M$ the hook isotypic components of $M$ can 
be recovered from the sign-isotypic components of $M \otimes \bigoplus_{i=0}^{n-1} S^{(n-i,1^i)}$, as
\begin{equation}
\begin{aligned} \label{eqn:hook-sign-iso}
\Hom_{\S_n}\left(\bigoplus_{i=0}^{n-1} S^{(n-i,1^i)},M\right) 
&\cong
\Hom_{\S_n}\left( \bigoplus_{i=0}^{n-1}S^{(i+1,1^{n-1-i})},M\right) 
\cong
\Hom_{\S_n}\left( \epsilon \otimes \bigoplus_{i=0}^{n-1}S^{(n-i,1^i)},M\right) \\
&\cong\Hom_{\S_n}\left( \epsilon, M \otimes \bigoplus_{i=0}^{n-1} S^{(n-i,1^i)}\right)
\cong \Hom_{\S_n}\left( \epsilon, M \otimes R_n^{(0,1)}\right).
\end{aligned}
 \end{equation}

The following proposition is a graded, character-level version of Equation~\eqref{eqn:hook-sign-iso}, for an arbitrary multigraded $\mathfrak{S}_n$-module $M$.

\begin{proposition}\label{prop:HS-generalization}
Let $m,n \geq 1$ and let $M$ be an $m$-multigraded $\mathfrak{S}_n$-module. Then
\begin{equation*} 
\left\langle \Frob(M \otimes R_n^{(0,1)}; q_1,\ldots, q_m;a), e_n\right\rangle = \sum_{k=0}^{n-1} a^k\left\langle \Frob(M; q_1,\ldots, q_m), s_{(k+1,1^{n-k-1})} \right\rangle.
\end{equation*}
\end{proposition}
Before we give a proof, we record the important special case where $M = R_n^{(m,0)}$. 
The final equality at $m=2$ follows from Theorem~\ref{thm:haiman} and Equation~\eqref{eq:Schroder-polynomial-def-second}; this case appears in \cite[proof of Theorem 4.11]{HaglundSergel}, and our proof of Proposition~\ref{prop:HS-generalization} follows a similar strategy. 

\begin{corollary}\label{cor:(2,0)(0,1)bound} For any $m,n \geq 1$,
\begin{equation*} 
\left\langle \Frob(R_n^{(m,0)} \otimes R_n^{(0,1)}; q_1,\ldots,q_m;a), e_n\right\rangle = \sum_{k=0}^{n-1} a^k\left\langle \Frob(R_n^{(m,0)}; q_1,\ldots,q_m), s_{(k+1,1^{n-k-1})} \right\rangle.
\end{equation*}
In particular, for $m=2$,
\begin{equation*} 
\left\langle \Frob(R_n^{(2,0)} \otimes R_n^{(0,1)}; q,t;a), e_n\right\rangle = \sum_{k=0}^{n-1} a^k\left\langle \nabla e_n, s_{(k+1,1^{n-k-1})} \right\rangle = \widetilde{S}_n(q,t,a).
\end{equation*}
\end{corollary}

\begin{proof}[Proof of Proposition~\ref{prop:HS-generalization}]
    As noted in \cite[Equation (4.9)]{HaglundSergel}, following from \cite{Bessenrodt}, for any $\mathfrak{S}_n$-modules $M$ and $N$ tensored under the diagonal action of $\mathfrak{S}_n$, we have that
    \begin{equation*} \Frob(M \otimes N) = \sum_{\nu \vdash n} s_\nu \sum_{\lambda, \mu \vdash n} \left\langle \Frob(M), s_\lambda \right\rangle \left\langle \Frob(N), s_\mu \right\rangle \left\langle s_\lambda * s_\mu, s_\nu \right\rangle,
    \end{equation*}
    where $\left\langle s_\lambda * s_\mu, s_\nu \right\rangle = g(\lambda, \mu, \nu)$ are the Kronecker coefficients. 
    In the present case, by keeping track of the gradings,
    \begin{equation*}
    \begin{aligned} 
    &\left\langle \Frob(M \otimes R_n^{(0,1)}; q_1,\ldots,q_m;a), e_n\right\rangle\\ 
    &\qquad=  \sum_{\lambda, \mu \vdash n} \Big\langle \Frob(M;q_1,\ldots,q_m), s_\lambda \Big\rangle \left\langle \Frob(R_n^{(0,1)};a), s_\mu \right\rangle g(\lambda, \mu, 1^n)  \\
    &\qquad=  \sum_{\lambda \vdash n} \Big\langle \Frob(M;q_1,\ldots,q_m), s_{\lambda'} \Big\rangle \left\langle \Frob(R_n^{(0,1)};a), s_{\lambda} \right\rangle,
    \end{aligned}
    \end{equation*}
    since $g(\lambda, \mu, 1^n) = \delta_{\mu, \lambda'}$.  
    Only Schur functions indexed by hooks appear in $\Frob(R_n^{(0,1)};a)$, so we reduce to
    \begin{equation*}
    \begin{aligned} \left\langle \Frob(M \otimes R_n^{(0,1)}; q_1,\ldots,q_m;a),e_n\right\rangle 
    &=  \sum_{k=0}^{n-1} a^k \left\langle \Frob(M;q_1,\ldots,q_m), s_{(k+1,1^{n-k-1})} \right\rangle. 
    \end{aligned}
    \end{equation*}
\end{proof}

\begin{lemma}\label{lem:quotient}
The natural surjection $\rho: R_n^{(2,0)} \otimes R_n^{(0,1)} \twoheadrightarrow R_n^{(2,1)}$ is a map of triply-graded $\S_n$-modules.
In particular, $R_n^{(2,1)}$ is a quotient $\mathfrak{S}_n$-module of $R_n^{(2,0)} \otimes R_n^{(0,1)}$.
\end{lemma}

\begin{proof}
    $I_n^{(2,1)}$ contains both the ideals generated by $I_n^{(2,0)}$ and $I_n^{(0,1)}$. 
    Their sum is the kernel of the natural surjection $\C[\bm x,\bm y,\bm\theta] \twoheadrightarrow R_n^{(2,0)}\otimes R_n^{(0,1)}$, so the projection $\C[\bm x,\bm y,\bm\theta]\twoheadrightarrow R_n^{(2,1)}$ factors through $R_n^{(2,0)}\otimes R_n^{(0,1)}$.
\end{proof}

It is important to note that $R_n^{(2,1)}$ is a proper quotient module of $R_n^{(2,0)} \otimes R_n^{(0,1)}$. 
In fact, comparing the computations of Zabrocki~\cite{Zabrocki2019}, which verify Conjecture~\ref{conj:zabrocki} for $n \leq 6$, with the Frobenius series of $R_n^{(2,0)} \otimes R_n^{(0,1)}$ computed via the Kronecker product formula in the proof of
Proposition~\ref{prop:HS-generalization}, one checks that for $2 \leq n \leq 6$, the sign character is the only irreducible character that has the same multiplicity in both $R_n^{(2,1)}$ and in $R_n^{(2,0)} \otimes R_n^{(0,1)}$.
Later in Section~\ref{subsec:main theorem} we will prove the sign-isotypic components are in fact always isomorphic. 

\begin{example}\label{ex:sign-tensor}
At $n=3$, one has 
\begin{align*}
    \Frob(R_3^{(2,1)}; q,t;a) = {}&\big(q^3+q^2t+qt^2+t^3+qt+ a(q^2+qt+t^2+q+t) +a^2\big)s_{(1,1,1)}\\ 
    &+ \big(q^2+qt+t^2+q+t + a(q+t+1)\big)s_{(2,1)}\\ 
    &+ s_{(3)},
\end{align*}
whereas
\begin{align*}
    \Frob&(R_3^{(2,0)} \otimes R_3^{(0,1)}; q,t;a) \\
    = {}&\big(q^3+q^2t+qt^2+t^3+qt+ a(q^2+qt+t^2+q+t) +a^2\big)s_{(1,1,1)}\\
    & + \big(q^2+qt+t^2+q+t +  a(q^3+q^2t+qt^2+t^3+q^2+2qt+t^2+q+t+1)\\ 
    & \quad{}+ a^2(q^2+qt+t^2+q+t)\big)\,s_{(2,1)}\\
    &+ \big(1 + a(q^2+qt+t^2+q+t) + a^2 (q^3+q^2t+qt^2+t^3+qt) \big) s_{(3)}.
\end{align*}
Note that the only character at which they coincide is the sign character. 
In this example, the sign character has multiplicity 11, which equals the number of little Schr\"oder paths $\widetilde{S}_3$ seen in Figure~\ref{fig:little-Schroder-3}.
\end{example}

\section{Relating hook and sign components via harmonic spaces}\label{sec:main}

In this section we present our main results: Theorem~\ref{thm:main}, which proves the sign-isotypic components of $R_n^{(2,1)}$ and $R_n^{(2,0)} \otimes R_n^{(0,1)}$ are isomorphic, and Corollary \ref{cor:sign-character-schroder}, which shows the triply-graded multiplicity of this sign character is the little Schr\"oder polynomial $\widetilde{S}_n(q,t,a)$.

To do so, we pass to the isomorphic space of harmonics $H_n^{(2,1)} \cong R_n^{(2,1)}$, and construct, for each $d$, a family of linearly independent alternating harmonics of $\bm{\theta}$-degree $d$. The construction has three steps.

First, we start with certain harmonics in $H_n^{(2,0)}$ that are symmetric in some variables and antisymmetric in the remainder. Second, we multiply each harmonic by a product of fermionic variables, thus increasing the $\bm{\theta}$-degree,
and then antisymmetrize the result to ensure it lies in the sign-isotypic component of $H_n^{(2,1)}$.
Lastly, we check that the resulting elements are harmonic and linearly independent. 

This construction then enables us to define an injective, grading-preserving map from the sign-isotypic component of $R_n^{(2,0)} \otimes R_n^{(0,1)}$ into that of $R_n^{(2,1)}$, which combined with Lemma~\ref{lem:quotient} proves the desired isomorphism.
\medskip

We begin by recalling the essential definitions on the harmonic spaces.
For any $r,s,\varepsilon \in \Z_{\geq 0}$ define
\begin{equation*}
    D_{r,s,\varepsilon} = \sum_{i=1}^n \partial_{x_i}^r \partial_{y_i}^s \partial_{\theta_i}^\varepsilon.
\end{equation*}
Here, we take the partial derivative $\partial_{\theta_i}$ on a monomial in $\theta$'s by first commuting the $\theta_i$ to the leftmost position, collecting signs along the way, and then differentiating as usual via the Leibniz rule.
Observe that $\partial_{\theta_i}^\varepsilon = 0$ whenever $\varepsilon \geq 2$.

\begin{definition} We define three important spaces:

\noindent
(1) The \newword{space of diagonal harmonics} is defined by
\begin{equation}\label{eq:diag-coinv}
    H_n^{(2,0)} = \left\{ f \in \C[\bm{x}, \bm{y}]  : D_{r,s,0}(f) = 0 \text{ for all } r,s \in \mathbb{Z}_{\geq 0}, r+s \geq 1 \right\}.
\end{equation}
(2) The \newword{space of diagonal super harmonics} is defined by
\begin{equation}\label{eq:diag-super-coinv}
    H_n^{(2,1)} = \left\{ f \in \C[\bm{x}, \bm{y}, \bm{\theta}]  : D_{r,s,\varepsilon}(f) = 0 \text{ for all } r,s,\varepsilon \in \mathbb{Z}_{\geq 0}, r+s+\varepsilon \geq 1 \right\}.
\end{equation}
(3) The \newword{space of fermionic harmonics} is defined by
\begin{equation*}
    H_n^{(0,1)} = \left\{ f \in \C[\bm{\theta}]  : D_{0,0,1}(f) = 0\right\}.
\end{equation*}
\end{definition}

\begin{remark}
As noted in the introduction, we will not use Haiman's operator theorem \cite{Haiman1994,Haiman2002}. 
Instead, the elements of $H_n^{(2,0)}$ we need are constructed explicitly.
\end{remark}

We will also use the following result. 
The case of $R_n^{(2,0)}$ is due to \cite{Haiman1994}. 
The case of $R_n^{(2,1)}$ is very similar to that of $R_n^{(1,1)}$ in \cite{SwansonWallach1} (which covers the case of $R_n^{(0,1)}$), and was proven in \cite[Theorem A.4]{JiangLentfer} in full generality.

\begin{proposition}\label{prop:H-R-isomorphism}
\noindent
(1) The projection $\C[\bm{x},\bm{y}] \to R_n^{(2,0)}$ restricts to an isomorphism of bigraded $\mathfrak{S}_n$-modules:
    \begin{equation*}
    \eta_{(2,0)}:H_n^{(2,0)} \xrightarrow{\ \sim\ } R_n^{(2,0)}, \quad \text{ given by } \quad h \mapsto h + I_n^{(2,0)}.
\end{equation*}
(2) The projection $\C[\bm{x},\bm{y},\bm{\theta}] \to R_n^{(2,1)}$ restricts to an isomorphism of triply-graded $\mathfrak{S}_n$-modules:
    \begin{equation*}
   \eta_{(2,1)}: H_n^{(2,1)} \xrightarrow{\ \sim\ } R_n^{(2,1)}, \quad \text{ given by } \quad h \mapsto h + I_n^{(2,1)}.
\end{equation*}
(3) The projection $\C[\bm{\theta}] \to R_n^{(0,1)}$ restricts to an isomorphism of graded $\mathfrak{S}_n$-modules:
    \begin{equation*}
   \eta_{(0,1)}: H_n^{(0,1)} \xrightarrow{\ \sim\ } R_n^{(0,1)}, \quad \text{ given by } \quad h \mapsto h + I_n^{(0,1)}.
\end{equation*}
\end{proposition}

Let $[n]:=\{1,\ldots, n\}$. For $B \subseteq [n]$ any nonempty subset, $r,s \in \Z_{\geq 0}$, and $\varepsilon \in \{0,1\}$, define the \newword{(super) polarized power-sum} symmetric function $p_{r,s,\varepsilon}(\bm{x}_B,\bm{y}_B, \bm\theta_B)$ as the sum:
\[
p_{r,s,\varepsilon}(\bm{x}_B,\bm{y}_B, \bm\theta_B) := \sum_{\ell \in B} x_\ell^r y_\ell^s \theta_\ell^{\varepsilon}. 
\]
In the special cases when $B = [n]$ or $\varepsilon =0$, we simplify notation and write
\[
p_{r,s,\varepsilon}(\bm{x},\bm{y},\bm{\theta})=\sum_{i=1}^n x_i^r y_i^s \theta_i^\varepsilon
\qquad \text{and} \qquad
p_{r,s}(\bm{x}_B,\bm{y}_B):=
p_{r,s,0}(\bm{x}_B,\bm{y}_B,\bm{\theta}_B).\]

For $B \subseteq [n]$, denote its Young subgroup in $\S_n$ by $\mathfrak{S}_B$. The \newword{antisymmetrizer} and \newword{symmetrizer} of $\mathfrak{S}_B$, respectively, are defined by
\[
     \cY_B= \frac{1}{|B|!}\sum_{\sigma \in \mathfrak{S}_B} \sgn(\sigma) \sigma
     \qquad \text{and} \qquad \cC_B= \frac{1}{|B|!}\sum_{\sigma \in \mathfrak{S}_B} \sigma. 
\]

We say an element $f$ in an $\S_n$-module $M$ is \newword{$\S_B$-antisymmetric} or \newword{$\S_B$-symmetric}, respectively, if $\sigma(f)=\sgn(\sigma)f$ or $\sigma(f) = f$ for all $\sigma \in \S_B$, or equivalently if
\[
\cY_B(f) = f \qquad \text{or} \qquad \cC_B(f) =f.
\]

Let $\overline{B}= [n]\setminus B$. For any $\S_n$-module $M$, the \newword{$\S_B$-antisymmetric, $\S_{\overline{B}}$-symmetric subspace of $M$} is the vector space:
\[
M_B:=\cY_B \cC_{\overline{B}} (M) = \cC_{\overline{B}} \cY_B(M) = \{ f \in M :  \cY_{B} (f)=f \text{ and } \cC_{\overline{B}}(f)=f \}.
\]
Note that since $\cY_{B}$ and $\cC_{\overline{B}}$ are idempotents, then $\cY_B \cC_{\overline{B}} (M)= \cY_{B}(M) \cap \cC_{\overline{B}}(M)$. 
In general, $\cY_B \cC_{\overline{B}} (M)$ is not an $\mathfrak{S}_n$-submodule of $M$.

We record some helpful lemmas.

\begin{lemma}\label{lem:Y_to_C}
Let $M$ be a $\C$-algebra on which $\S_n$ acts by algebra automorphisms. 
    For $f,g \in M$ and $B \subseteq [n]$, if $f$ is $\S_B$-antisymmetric, then
    \[
    \cY_B(fg) = f \; \cC_B(g) \qquad \text{and} \qquad \cC_B(fg) = f \; \cY_B(g) .
    \]
\end{lemma}
\begin{proof}
Since $\sigma(f) = \sgn(\sigma) f$ for all $\sigma \in \S_B$, then
    \[
    \cY_B(fg)
    = \frac{1}{|B|!}\sum_{\sigma \in \S_B} \sgn(\sigma) \sigma(f)\sigma(g)
    = f \frac{1}{|B|!} \sum_{\sigma \in \S_B} \sigma(g) = f \; \cC_B(g),
    \]
    and
     \[
    \cC_B(fg)
    = \frac{1}{|B|!}\sum_{\sigma \in \S_B}  \sigma(f)\sigma(g)
    = f \frac{1}{|B|!} \sum_{\sigma \in \S_B} \sgn(\sigma)\sigma(g) = f \; \cY_B(g).
    \]
\end{proof}

\begin{lemma}\label{lem:commuting-operators}
Fix $n \geq 1$ and let $B \subseteq [n]$.
    For any nonnegative integers $r,s,\varepsilon$, the following operators commute:
   \begin{equation*}
        D_{r,s,\varepsilon} \cY_{B} = \cY_{B}D_{r,s,\varepsilon}
        \qquad \text{and}
        \qquad
        D_{r,s,\varepsilon} \cC_{B} = \cC_{B}D_{r,s,\varepsilon}.
    \end{equation*}

\end{lemma}

\begin{proof}
Observe that for all $\sigma \in \mathfrak{S}_n$ and for all $i \in \{1,\ldots,n\}$, 
\[
\sigma \partial_{x_i} \sigma^{-1} = \partial_{x_{\sigma(i)}},
\quad 
\sigma \partial_{y_i} \sigma^{-1} = \partial_{y_{\sigma(i)}},
\quad \text{and}\quad
\sigma \partial_{\theta_i} \sigma^{-1} = \partial_{\theta_{\sigma(i)}}.
\]
Thus we may write
\begin{equation*}
    \sigma (\partial_{x_i}^r \partial_{y_i}^s \partial_{\theta_i}^\varepsilon ) \sigma^{-1} =  (\sigma \partial_{x_i}^r  \sigma^{-1} )(\sigma  \partial_{y_i}^s \sigma^{-1} )( \sigma \partial_{\theta_i}^\varepsilon  \sigma^{-1}) = \partial_{x_{\sigma(i)}}^r\partial_{y_{\sigma(i)}}^s\partial_{\theta_{\sigma(i)}}^\varepsilon.
\end{equation*}
Then by summing over all $i$ and by reindexing via $j = \sigma(i)$, we obtain
\begin{equation*}
    \sigma D_{r,s,\varepsilon} \sigma^{-1} 
    = \sum_{i=1}^n \partial_{x_{\sigma(i)}}^r \partial_{y_{\sigma(i)}}^s \partial_{\theta_{\sigma(i)}}^\varepsilon 
    = \sum_{j=1}^n \partial_{x_j}^r \partial_{y_j}^s \partial_{\theta_j}^\varepsilon = D_{r,s,\varepsilon}.
\end{equation*}
Hence $\sigma D_{r,s,\varepsilon} = D_{r,s,\varepsilon} \sigma$ for any $\sigma \in \S_n$ and we conclude that
\begin{equation*}
|B|! D_{r,s,\varepsilon} \cY_{B} = \sum_{\sigma \in \mathfrak{S}_B} D_{r,s,\varepsilon} \sgn(\sigma) \sigma  = \sum_{\sigma \in \mathfrak{S}_B} \sgn(\sigma) \sigma D_{r,s,\varepsilon}=  |B|! \cY_{B} D_{r,s,\varepsilon}
\end{equation*}
and similarly,
\begin{equation*}
|B|! D_{r,s,\varepsilon} \cC_{B} = \sum_{\sigma \in \mathfrak{S}_B} D_{r,s,\varepsilon}  \sigma  = \sum_{\sigma \in \mathfrak{S}_B}  \sigma D_{r,s,\varepsilon}=  |B|! \cC_{B} D_{r,s,\varepsilon}.
\end{equation*}
\end{proof}

\subsection{Identifying hook-isotypic components}
\label{subsec:identifyhook}
Fix an integer $d$ with $0 \leq d \leq n-1$ and define the following sets: 
\begin{equation*}
   \beta_d := \{1,\ldots,n-d\}, \quad \text{and} \quad \overline{\beta_d} := [n]\setminus \beta_d=\{n-d+1,\ldots,n\}.
\end{equation*}
When the fixed choice of $d$ is clear, we write $\beta := \beta_d$ and $\overline{\beta} := \overline{\beta_d}$.
Then $\S_\beta \cong \S_{n-d} \times \S_1^{ d}$ and $\S_{\overline{\beta}} \cong \S_1^{n-d} \times \S_d$.

For each $0 \leq d \leq n-1$ and any $\mathfrak{S}_n$-module $M$,
consider the
\newword{$\S_{n-d}$-antisymmetric, $\S_d$-symmetric linear subspace}
\begin{align*}
M_{\beta_d}&= \cY_{\beta_d} \cC_{\overline{\beta_d}}(M)
= 
\{ f \in M : (\sigma,\tau)f = \sgn(\sigma) f \text{ for all } (\sigma,\tau) \in \S_{n-d} \times \S_d\} 
\subset M \end{align*}
and let 
\[
\widetilde{M}_{\beta_d} := \{ f \in M_{\beta_d} : \cY_{\beta_d \cup \{i\}}(f) =0 \text{ for all } i \in \overline{\beta_d}\}.
\]
Since we fix $d$ throughout, abbreviate $M_\beta := M_{\beta_d}$ and $\widetilde{M}_\beta := \widetilde{M}_{\beta_d}$.
Note that when $d=0$, we have $\beta_0 = [n]$ and $\overline{\beta_0} = \varnothing$, so $\widetilde{M}_{[n]} = M_{[n]}$.

The following isomorphism is straightforward, but we include a proof for completeness.

\begin{proposition}\label{prop:hook-hom-iso}
    For any $\mathfrak{S}_n$-module $M$, there is an isomorphism of vector spaces
    \begin{align*}
M_\beta &\cong \Hom_{\mathfrak{S}_n} (S^{(d,1^{n-d})} \oplus S^{(d+1,1^{n-d-1})}, M),
    \end{align*}
    with the convention that $S^{(0,1^n)}$ is $0$.
Thus, letting $\left\langle e_- \right\rangle = (S^{(d,1^{n-d})})_\beta$ and $\left\langle e_+ \right\rangle = (S^{(d+1,1^{n-d-1})})_\beta$, we have a decomposition, 
\begin{align*}
M_\beta 
&\cong \left\langle e_- \right\rangle  \otimes \Hom_{\S_n}( S^{(d,1^{n-d})}, M) \; \oplus \; \left\langle e_+ \right\rangle  \otimes \Hom_{\S_n}( S^{(d+1,1^{n-d-1})}, M).
    \end{align*}
\end{proposition}

\begin{proof}
Recall that by the Littlewood--Richardson rule, for any $0 \leq d \leq n-1$, we have a decomposition into Specht modules, 
\[
\Ind_{\mathfrak{S}_{n-d} \times \mathfrak{S}_d}^{\mathfrak{S}_n}(\epsilon_{n-d} \boxtimes \mathbbm{1}_d) \cong \Ind_{\mathfrak{S}_{n-d} \times \mathfrak{S}_d}^{\mathfrak{S}_n}(S^{(1^{n-d})} \boxtimes S^{(d)})\cong S^{(d,1^{n-d})} \oplus S^{(d+1,1^{n-d-1})},
\]
where $\epsilon$ and $\mathbbm{1}$ denote the corresponding sign and trivial representations, respectively, and $\boxtimes$ the outer tensor product. Since by definition 
\[
M_\beta = \{ f \in M \; :\; (\sigma,\tau)f = \sgn(\sigma)f \;\text{for all} \;    (\sigma,\tau) \in \mathfrak{S}_{n-d} \times \mathfrak{S}_d\},
\]
then by Frobenius reciprocity, it is equivalent to show that
\begin{align*}
M_\beta
 &\cong 
 \Hom_{\S_n} (
\Ind_{\mathfrak{S}_{n-d} \times \mathfrak{S}_d}^{\mathfrak{S}_n}(\epsilon_{n-d} \boxtimes \mathbbm{1}_d), M)\\
&\cong
\Hom_{\mathfrak{S}_{n-d} \times \mathfrak{S}_d} (
\epsilon_{n-d} \boxtimes \mathbbm{1}_d, \Res_{\mathfrak{S}_{n-d} \times \mathfrak{S}_d}^{\mathfrak{S}_n}M)
= \Hom_{\mathfrak{S}_{n-d} \times \mathfrak{S}_d} (
\epsilon_{n-d} \boxtimes \mathbbm{1}_d, M).
\end{align*}

Now, since $\epsilon \boxtimes \mathbbm{1} \cong \C \otimes \C$ then any $\Psi \in \Hom_{\mathfrak{S}_{n-d} \times \mathfrak{S}_d} (
\epsilon_{n-d} \boxtimes \mathbbm{1}_d, M)$ is uniquely determined by its image $\Psi(1\otimes 1) \in M$. Since $\Psi$ is an $\mathfrak{S}_{n-d} \times \mathfrak{S}_d$-module homomorphism we have that for any $(\sigma,\tau) \in \mathfrak{S}_{n-d} \times \mathfrak{S}_d$, 
\[
(\sigma,\tau)\Psi(1\otimes 1)=
\Psi((\sigma,\tau)(1\otimes 1)) = \Psi(\sgn(\sigma)1\otimes 1) = \sgn(\sigma)\Psi(1\otimes 1).
\]
Thus $\Psi(1\otimes 1) \in M_\beta$ for all $\Psi$.

Conversely, given any $f \in M_\beta$ let $\Psi_f: \epsilon_{n-d} \boxtimes \mathbbm{1}_d  \to M$ be defined via $\Psi_f(1 \otimes 1) =f$. Evidently, $\Psi_f$ is a linear map such that for any $(\sigma,\tau) \in \mathfrak{S}_{n-d} \times \mathfrak{S}_d$, 
\[
(\sigma,\tau)\Psi_f(1 \otimes 1) = (\sigma,\tau)f = \sgn(\sigma) f = \sgn(\sigma)  \Psi_f(1 \otimes 1) =   \Psi_f(\sgn(\sigma)(1 \otimes 1)) = \Psi_f((\sigma,\tau)(1\otimes 1)).
\]
Thus $\Psi_f$ is an $\mathfrak{S}_{n-d} \times \mathfrak{S}_d$-module homomorphism, as desired. 

By semisimplicity, the evaluation map
\[
\bigoplus_{\lambda} S^\lambda \otimes M_\lambda \longrightarrow M, \qquad v \otimes \phi \mapsto \phi(v)
\]
is an isomorphism of $\S_n$-modules, where $M_\lambda = \Hom_{\S_n}(S^\lambda, M)$ is the \emph{multiplicity space} of $\lambda$, and $\S_n$ acts on the first tensor factor. 
Each $\phi \in M_\lambda$ commutes with $\cY_\beta \cC_{\overline{\beta}}$, so applying this to both sides gives $M_\beta \cong \bigoplus_\lambda (S^\lambda)_\beta \otimes M_\lambda$.
Since by the first claim we have that 
\[
(S^\lambda)_\beta \cong \begin{cases}
    \C & \text{if } \lambda= (d,1^{n-d}) \text { or } (d+1, 1^{n-d-1}) \\
    0 & \text{otherwise},
\end{cases}
\]
then letting $e_-$ and $e_+$ denote the respective generators when $\lambda= (d,1^{n-d}), (d+1, 1^{n-d-1})$, we obtain
\[
M_\beta \cong 
\left\langle e_- \right\rangle \otimes \Hom_{\S_n}( S^{(d,1^{n-d})}, M) \; \oplus \; \left\langle e_+ \right\rangle \otimes \Hom_{\S_n}( S^{(d+1,1^{n-d-1})}, M).
\]
\end{proof}

We establish some results on the antisymmetrizer $\cY_{\beta \cup \{i\}}$.
\begin{lemma}\label{lem:simplified-criteria}
Let $ 1 \leq d \leq n-1$.
    If $f \in M_\beta$ and $\cY_{\beta \cup \{i_0\}} f =0$ for a single $i_0 \in \overline{\beta}$, then $\cY_{\beta \cup \{i\}} f = 0$ for all $i \in \overline{\beta}$. 
    Hence 
    \begin{equation*}
        \widetilde{M}_\beta = \{ f \in M_\beta  : \text{there exists } i_0 \in \overline{\beta} \text{ such that } \cY_{\beta \cup \{i_0\} }f = 0\}.
    \end{equation*}
\end{lemma}
\begin{proof}
Fix $i_0 \in \overline{\beta}$.
    For $i \in \overline{\beta}$, let $\tau$ denote the transposition $(i_0, i) \in \mathfrak{S}_{\overline{\beta}}$.
    Then $\tau$ fixes $\beta$ pointwise and $\tau(\beta \cup \{i_0\}) = \beta \cup \{i\}$, so $\cY_{\beta \cup \{i\}} = \tau \cY_{\beta \cup \{i_0\}} \tau^{-1}$.
Since $\cC_{\overline{\beta}}(f) = f$, then $\tau^{-1} f=f$, so $\cY_{\beta \cup \{i\}} f = \tau \cY_{\beta \cup \{i_0\}} f = 0$.
\end{proof}

\begin{corollary}\label{cor:antisymmetrizers-hook}
Let $1 \leq d \leq n-1$ and  $i_0 \in \overline{\beta}$. Then $\cY_{\beta \cup \{i_0\}}(\left\langle e_+ \right\rangle)=0$ and $\cY_{\beta \cup \{i_0\}}(\left\langle e_- \right\rangle)\cong \left\langle e_- \right\rangle$.
\end{corollary}

\begin{proof}
    Without loss of generality, by Lemma~\ref{lem:simplified-criteria}, assume $i_0 =n-d+1$ and set $\alpha = \beta \cup \{n-d+1\}$. Then, as in the proof of Proposition~\ref{prop:hook-hom-iso}, we have that for any $\S_n$-module $M$,
    \[
    \cY_\alpha(M) \cong \Hom_{\S_n}\left(\Ind_{\S_{n-d+1} \times \S_1^{d-1}}^{\S_n}(\epsilon_{n-d+1} \boxtimes \mathbbm{1}_1^{d-1}), M\right). 
    \] 
    Repeated application of the Pieri rule gives
    \[
\cY_\alpha(M) \cong \Hom_{\S_n}\left( \bigoplus_{\nu \supseteq (1^{n-d+1})} (S^\nu)^{\oplus f^{\nu/(1^{n-d+1})}}, M \right)
    \]
    where the sum is taken over all partitions $\nu$ of size $n$ containing the column shape $(1^{n-d+1})$, each with multiplicity $f^{\nu/(1^{n-d+1})}$, the number of standard Young tableaux of skew shape $\nu/(1^{n-d+1})$. 
Since $(d+1,1^{n-d-1})$ has only $n-d$ parts, and since $(d,1^{n-d})/(1^{n-d+1})$ is a single row of $d-1$ boxes, by Schur's lemma it follows that 
$\cY_\alpha(S^{(d+1,1^{n-d-1})})=0$ and $\cY_\alpha(S^{(d,1^{n-d})})\cong \C$. 
Since $\left\langle e_+ \right\rangle \subseteq S^{(d+1,1^{n-d-1})}$, then $\cY_\alpha(\left\langle e_+ \right\rangle )=0$. 

Let $\{u_1,\ldots,u_n\}$ denote the standard basis for $\C^n$. 
One can check that $S^{(d,1^{n-d})} \oplus S^{(d+1,1^{n-d-1})} \cong \bigwedge^{n-d} \C^n$ with basis $\bm{u}_J:= u_{j_1} \wedge \dots \wedge u_{j_{n-d}}$ for subsets $J = \{j_1<\dots< j_{n-d}\} \subseteq [n]$.
Consider $\bm{u}_\beta = u_1 \wedge \dots \wedge u_{n-d}$. 
Since $\cY_\beta(\bm{u}_\beta) = \bm{u}_\beta$ and $\cC_{\overline{\beta}}(\bm{u}_\beta)= \bm{u}_\beta$, it follows that $\bm{u}_\beta \in (S^{(d,1^{n-d})} \oplus S^{(d+1,1^{n-d-1})})_\beta$. 
Every element of $\S_\alpha$ factors as $\sigma$ or as $(n-d+1,k)\sigma$ for some $\sigma \in \S_\beta$ and $k \in \beta$, so
\begin{align*}
\cY_\alpha(\bm{u}_\beta) &= \frac{1}{(n-d+1)!}\sum_{\sigma \in \S_\beta} \left(\bm{u}_\beta - \sum_{k=1}^{n-d}(n-d+1,k)\bm{u}_\beta\right)\\
&=
\frac{1}{n-d+1}\left(\bm{u}_\beta - \sum_{k=1}^{n-d}(n-d+1,k)\bm{u}_\beta\right).
\end{align*}
For $k \in \beta$, note that $(n-d+1,k)\bm{u}_\beta = u_1 \wedge \cdots \wedge u_{k-1} \wedge u_{n-d+1} \wedge u_{k+1} \wedge \cdots \wedge u_{n-d}$ is, up to sign, the basis vector $\bm{u}_J$ for $J = (\beta \setminus \{k\}) \cup \{n-d+1\} \neq \beta$.
Since the $\bm{u}_J$ form a basis, the coefficient of $\bm{u}_\beta$ in $\cY_\alpha(\bm{u}_\beta)$ is $\frac{1}{n-d+1} \neq 0$, so $\cY_\alpha(\bm{u}_\beta) \neq 0$.

Since $\cY_\beta \cC_{\overline{\beta}}$ acts componentwise on direct sums, $(S^{(d,1^{n-d})} \oplus S^{(d+1,1^{n-d-1})})_\beta = \left\langle e_-\right\rangle \oplus \left\langle e_+ \right\rangle$, so we may write $\bm{u}_\beta = \bm{u}_\beta^- + \bm{u}_\beta^+$ with $\bm{u}_\beta^- \in \left\langle e_- \right\rangle$ and $\bm{u}_\beta^+ \in \left\langle e_+ \right\rangle$.
Since $\cY_\alpha(\left\langle e_+ \right\rangle) = 0$, we have $\cY_\alpha(\bm{u}_\beta^-) = \cY_\alpha(\bm{u}_\beta) \neq 0$.
In particular $\bm{u}_\beta^- \neq 0$, so $\bm{u}_\beta^-$ spans the line $\left\langle e_- \right\rangle$, and thus $\cY_\alpha$ is injective on $\left\langle e_- \right\rangle$, that is, $\cY_\alpha(\left\langle e_- \right\rangle) \cong \left\langle e_- \right\rangle$.
\end{proof}

\begin{theorem}\label{thm:tildehook-hom-iso}
For any $\S_n$-module $M$ and $i \in \overline{\beta}$, we have $\cY_{\beta \cup \{i\}}(M_\beta) \cong \Hom_{\S_n}( S^{(d,1^{n-d})}, M)$. Furthermore,
    \[
\widetilde{M}_\beta \cong 
\Hom_{\S_n}( S^{(d+1,1^{n-d-1})}, M).
    \]
\end{theorem}
\begin{proof}

If $d = 0$, then $\overline{\beta_0} = \emptyset$, so $\widetilde{M}_{[n]} = M_{[n]} \cong \Hom_{\S_n}(S^{(1^n)}, M)$ by Proposition~\ref{prop:hook-hom-iso} (with $S^{(0,1^n)} = 0$). 

Now let $1 \leq d \leq n-1$, fix $i \in \overline{\beta}$, and let $\alpha = \beta \cup \{i\}$.
Consider $\cY_{\alpha}$ as a linear map on $M$.
By Proposition~\ref{prop:hook-hom-iso} and Corollary~\ref{cor:antisymmetrizers-hook}, 
\begin{align*}
\cY_{\alpha}(M_\beta) &\cong \cY_{\alpha}(\left\langle e_- \right\rangle ) \otimes \Hom_{\S_n}( S^{(d,1^{n-d})}, M) \; \oplus \; \cY_{\alpha}(\left\langle e_+ \right\rangle)   \otimes \Hom_{\S_n}( S^{(d+1,1^{n-d-1})}, M) \\
&\cong \left\langle e_- \right\rangle  \otimes \Hom_{\S_n}( S^{(d,1^{n-d})}, M). 
\end{align*}
Hence the restriction map $\cY_{\alpha}|_{M_\beta}: M_\beta \to \cY_{\alpha}(M_\beta)$ has kernel $\left\langle e_+ \right\rangle   \otimes \Hom_{\S_n}( S^{(d+1,1^{n-d-1})}, M) $. 
On the other hand, since by Lemma~\ref{lem:simplified-criteria}, $\widetilde{M}_\beta = \ker(\cY_{\alpha}|_{M_\beta})$, then it follows that
\[
\widetilde{M}_\beta \cong \left\langle e_+ \right\rangle   \otimes \Hom_{\S_n}( S^{(d+1,1^{n-d-1})}, M) \cong \Hom_{\S_n}( S^{(d+1,1^{n-d-1})}, M).
\]
\end{proof}

As a consequence of the proof above, the hook-isotypic components arise as both the image and the kernel of antisymmetrizers.

\begin{corollary}
    For any $1\leq d \leq n-1$ and $i \in \overline{\beta_d}$, we have that $\widetilde{M}_{\beta_{d-1}} \cong \cY_{\beta_d \cup \{i\}}(M_{\beta_d})$.
\end{corollary}

We apply these results directly to $M=H_n^{(2,0)}$. For convenience, set
\begin{align*}
G_{\beta_d} &:=\left(  H_n^{(2,0)}\right)_{\beta_d}
\qquad \text{ and } \qquad \widetilde{G}_{\beta_d} := \widetilde{\left(  H_n^{(2,0)}\right)}_{\beta_d},
\end{align*}
and, similarly to before, write $G_\beta := G_{\beta_d}$ and $\widetilde{G}_\beta := \widetilde{G}_{\beta_d}$.

The following corollary provides a combinatorial formula for the bigraded dimensions of the hook-isotypic components of 
$H_n^{(2,0)}$.

\begin{corollary}\label{cor:schroeder-hook-dim}
    For $0\leq d \leq n-1$, the bigraded Hilbert series of $G_{\beta_d}$ and $\widetilde{G}_{\beta_d}$ are given by
    \begin{align*}
\Hilb(G_{\beta_d};q,t) = S_{n,d}(q,t) \qquad \text{ and } \qquad \Hilb(\widetilde{G}_{\beta_d};q,t) = \widetilde{S}_{n,d}(q,t).    
    \end{align*}
\end{corollary}

\begin{proof}
 Computing the graded dimensions, by Proposition~\ref{prop:H-R-isomorphism}, Theorem~\ref{thm:haiman}, Proposition~\ref{prop:hook-hom-iso}, and Theorem~\ref{thm:tildehook-hom-iso}
 it follows that, 
\begin{align*}
\Hilb(G_\beta;q,t) &= 
\Hilb\left(\Hom_{\mathfrak{S}_n} \left(S^{(d,1^{n-d})} \oplus S^{(d+1,1^{n-d-1})}, H_n^{(2,0)}\right);q,t\right) \\
& = \left\langle s_{(d,1^{n-d})} + s_{(d+1,1^{n-d-1})}, \nabla e_n\right\rangle 
=
\Big\langle \nabla e_n, e_{n-d}h_d \Big\rangle
\end{align*}
and similarly for $\widetilde{G}_\beta$. The final equalities follow from Theorems~\ref{thm:qt-Schroeder-two-hooks} and \ref{thm:qt-Schroeder-one-hook}.
\end{proof}

\subsection{A harmonic lift of the hook-isotypic components}\label{subsec:harmoniclift}

Recall Rota's \newword{apolar form} for diagonal harmonics \cite[Section 1.3]{Haiman1994}: 
\begin{equation}\label{eq:Rota-apolar-form}
\left\langle f,g \right\rangle_A := f(\partial_{x_1}, \dots, \partial_{x_n}, \partial_{y_1}, \dots, \partial_{y_n}) g(x_1, \dots, x_n, y_1, \dots, y_n)|_{\bm{x}, \bm{y}=0}.
\end{equation}

Then, $\left\langle \cdot, \cdot \right\rangle_A$ is a symmetric, nondegenerate $\mathfrak{S}_n$-invariant form on $\C[\bm{x},\bm{y}]$ where differentiation is adjoint to multiplication, so that for any variable $z \in \bm{x} \cup \bm{y}$ and any $\sigma \in \S_n$:
\[
\left\langle z \cdot f, g \right\rangle_A = \left\langle f, \partial_z \cdot g\right\rangle_A \quad\text{ and }\quad \left\langle \sigma f, g \right\rangle_A = \left\langle f, \sigma^{-1} g \right\rangle_A.
\]
In particular, $\left\langle \cY_B(f),g \right\rangle_A = \left\langle f, \cY_B(g)\right\rangle_A$ for any $B \subseteq [n]$, so $\cY_B$ is self-adjoint.

Furthermore, from \cite[Section 1.3]{Haiman1994} the orthogonal complement of $I_n^{(2,0)}$ in $\C[\bm{x},\bm{y}]$ with respect to the apolar form is $(I_n^{(2,0)})^\perp = H_n^{(2,0)}$, so $\C[\bm{x},\bm{y}] = H_n^{(2,0)} \oplus I_n^{(2,0)}$.
Recalling explicit generators for the ideals $I_n^{(2,0)}$ and $I_n^{(2,1)}$ will be useful.
The diagonal case of the following result is Weyl's polarization theorem \cite{Weyl} and the super diagonal case is a special case of \cite[Theorem 17]{OrellanaZabrocki}.

\begin{proposition}\label{prop:ideal-generators}
    The (nonunital) algebra of diagonal invariants
    $\C[\bm{x}_B, \bm{y}_B]_+^{\mathfrak{S}_B}$ is generated by the polarized power sums
    \begin{equation}\label{eq:gen-set-A-20}
    \{ p_{r,s}(\bm{x}_B,\bm{y}_B) \;| \;
        r,s \in \Z_{\geq 0}, 1 \leq r+s \leq |B|\}.
\end{equation}
    Consequently, the set~\eqref{eq:gen-set-A-20} at $B = [n]$ generates the defining ideal $I_{n}^{(2,0)}$.

        The (nonunital) algebra of diagonal super invariants
    $\C[\bm{x}_B, \bm{y}_B, \bm{\theta}_B]_+^{\mathfrak{S}_B}$ is generated by the polarized power sums
    \begin{equation}\label{eq:gen-set-A-21}
\{ p_{r,s,\varepsilon}(\bm{x}_B,\bm{y}_B,\bm{\theta}_B) \;|\;
        r,s \in \Z_{\geq 0},\ \varepsilon \in \{0,1\}, 1 \leq r+s+\varepsilon \leq |B| \}.
\end{equation}
    Consequently, the set~\eqref{eq:gen-set-A-21} at $B = [n]$ generates the defining ideal $I_{n}^{(2,1)}$.
\end{proposition}

The following lemma is a generalization of classical results on Vandermonde determinants and alternating polynomials (see \cite[Chapter I.3]{Macdonald}).

\begin{lemma}\label{lem:generalized-Vandermonde}
Let $\alpha \sqcup {\overline{\alpha}}$ partition $\{1,\ldots,n\}$. 
Let $\mathcal A \subseteq \C[\bm{x},\bm{y}]$ denote the space of polynomials that are $\mathfrak{S}_\alpha$-antisymmetric and $\mathfrak{S}_{\overline{\alpha}}$-symmetric.
Let $\mathcal A_\alpha \subseteq \C[\bm{x}_\alpha,\bm{y}_\alpha]$ denote the space of polynomials, in the ${\alpha}$-indexed variables only, that are $\mathfrak{S}_{\alpha}$-antisymmetric.
For a list $\gamma = \left((c_1,e_1), \ldots, (c_{|{\alpha}|}, e_{|{\alpha}|}) \right)$ of pairs of nonnegative integers, define a generalized Vandermonde determinant by
\begin{equation*}
    \delta_\gamma := \det(x_\ell^{c_m} y_\ell^{e_m})_{\ell \in {\alpha}, 1 \leq m \leq |{\alpha}|} = \sum_{\sigma \in \mathfrak{S}_{\alpha}} \sgn(\sigma)\prod_{m=1}^{|{\alpha}|} x_{\sigma(\ell_m)}^{c_m} y_{\sigma(\ell_m)}^{e_m},
\end{equation*}
where ${\alpha} = \{\ell_1 < \cdots < \ell_{|{\alpha}|}\} $ (rows are indexed by the indices in ${\alpha}$ in increasing order, and columns are indexed by $m \in \{1,\ldots,|{\alpha}|\}$ from the bidegrees $(c_m,e_m)$).
Then 
\begin{equation*}
    \mathcal A_{\alpha} = \Span_\C \{ \delta_\gamma : \gamma \in \left(\Z_{\geq 0}^2\right)^{|{\alpha}|} \},
\end{equation*}
and
\begin{equation*}
    \mathcal A = \Span_\C \{ \delta_\gamma P : \gamma \in \left(\Z_{\geq 0}^2\right)^{|{\alpha}|}, \ P \in \C[\bm{x}_{\overline{\alpha}}, \bm{y}_{\overline{\alpha}}]^{\mathfrak{S}_{\overline{\alpha}}} \}.
\end{equation*}
\end{lemma}

\begin{proof}
    First, $\C[\bm{x}, \bm{y}] = \C[\bm{x}_{\alpha}, \bm{y}_{\alpha}] \otimes_\C \C[\bm{x}_{\overline{\alpha}}, \bm{y}_{\overline{\alpha}}]$ and the actions of $\mathfrak{S}_{\alpha}$ and $\mathfrak{S}_{\overline{\alpha}}$ commute. 
    Thus $\mathcal A$ decomposes as tensors of $\mathfrak{S}_{\alpha}$-antisymmetric polynomials with $\mathfrak{S}_{\overline{\alpha}}$-symmetric polynomials.

    It remains to identify the span of $\mathfrak{S}_{\alpha}$-antisymmetric polynomials $\mathcal{A}_{\alpha}$ with the span of the determinants $\delta_\gamma$.
    The monomials $\prod_{\ell \in {\alpha}} x_\ell^{a_\ell} y_\ell^{b_\ell}$ form a basis of $\C[\bm{x}_{\alpha},\bm{y}_{\alpha}]$, which $\mathfrak{S}_{\alpha}$ permutes.
    Thus the antisymmetrizations $\cY_{\alpha} \prod_{\ell \in {\alpha}} x_\ell^{a_\ell} y_\ell^{b_\ell}$ span $\mathcal{A}_{\alpha}$.
    Fix a monomial $\prod_{\ell \in {\alpha}} x_\ell^{a_\ell} y_\ell^{b_\ell}$ and let $\gamma$ be the list of bidegrees $(a_{\ell_1},b_{\ell_1}),\ldots, (a_{\ell_{|{\alpha}|}},b_{\ell_{|{\alpha}|}})$, where $\ell_1 < \cdots < \ell_{|{\alpha}|}$.
    By antisymmetrizing and using the Leibniz formula for the determinant, we have
    \begin{align}\label{eq:determinants}
        \sum_{\sigma \in \mathfrak{S}_{\alpha}} \sgn(\sigma) \sigma \left(\prod_{m=1}^{|{\alpha}|} x_{\ell_m}^{a_{\ell_m}}y_{\ell_m}^{b_{\ell_m}} \right) &= \sum_{\sigma \in \mathfrak{S}_{\alpha}} \sgn(\sigma) \prod_{m=1}^{|{\alpha}|} x_{\sigma(\ell_m)}^{a_{\ell_m}} y_{\sigma(\ell_m)}^{b_{\ell_m}} \\
        &= \det (x_\ell^{c_m} y_\ell^{e_m})_{\ell \in {\alpha}, 1 \leq m \leq |{\alpha}|} = \delta_\gamma. \nonumber 
    \end{align}
    Here, in the determinant, the row index becomes $\ell \in {\alpha}$ and the column index $m$ is associated to the bidegree $(c_m,e_m) = (a_{\ell_m}, b_{\ell_m})$.
    Thus each such antisymmetrized monomial is some $\delta_\gamma$, and conversely, each $\delta_\gamma$ is also by Equation~\eqref{eq:determinants} the antisymmetrization of a monomial. 
\end{proof}

We continue with technical lemmas that will be used to show harmonicity of the lift.
\begin{lemma}\label{lem:ideal-membership} Let $i \in \overline{\beta}$, let $\alpha = \beta \cup \{i\}$, and let $r,s \geq 0$ with $r+s \geq 1$. 
Then for any $A \in \cY_\alpha \cC_{\overline{\alpha}}(\C[\bm{x},\bm{y}])$, 
    \begin{equation*}\label{eq:claim-ideal-membership}
        \cC_{\overline{\beta}}(x_i^r y_i^s A) \in I_n^{(2,0)}. 
    \end{equation*}
    Thus, $\left\langle f, \cC_{\overline{\beta}}(x_i^r y_i^s A)\right\rangle_A =0$ for all $f \in H_n^{(2,0)}$.
\end{lemma}

\begin{proof}
    By Lemma~\ref{lem:generalized-Vandermonde}, the space of $\mathfrak{S}_{\alpha}$-antisymmetric and $\mathfrak{S}_{\overline{\alpha}}$-symmetric polynomials in $\C[\bm{x},\bm{y}]$ is spanned by products $\delta_\gamma P$,
    where $\delta_\gamma$ 
    is a generalized Vandermonde determinant 
    and $P \in \C[\bm{x}_{\overline{\alpha}},\bm{y}_{\overline{\alpha}}]^{\mathfrak{S}_{\overline{\alpha}}}$.
    By linearity, it suffices to address $A = \delta_\gamma (P+ m)$ with $P \in \C[\bm{x}_{\overline{\alpha}},\bm{y}_{\overline{\alpha}}]^{\mathfrak{S}_{\overline{\alpha}}}_+$ and $m \in \C$. 

    By Proposition~\ref{prop:ideal-generators}, the algebra $\C[\bm{x}_{\overline{\alpha}}, \bm{y}_{\overline{\alpha}}]_+^{\mathfrak{S}_{\overline{\alpha}}}$ is generated by the polarized power sums $p_{c,e}(\bm{x}_{\overline{\alpha}}, \bm{y}_{\overline{\alpha}})$ with $c + e \geq 1$.
    Thus $P = F\big(p_{c,e}(\bm{x}_{\overline{\alpha}}, \bm{y}_{\overline{\alpha}})\big)$ for some polynomial $F$ in finitely many variables with zero constant term.
    Because $\alpha$ and $\overline{\alpha}$ partition $[n]$, and $p_{c,e}(\bm{x},\bm{y}) \in I_n^{(2,0)}$ for $c+e \geq 1$, 
    \[p_{c,e}(\bm{x}_{\overline{\alpha}},\bm{y}_{\overline{\alpha}}) = p_{c,e}(\bm{x},\bm{y}) - p_{c,e}(\bm{x}_{\alpha},\bm{y}_{\alpha})
\equiv - p_{c,e}(\bm{x}_{\alpha},\bm{y}_{\alpha}) \pmod {I_n^{(2,0)}}.\]
Set $\widetilde{P} := F\big({-p_{c,e}(\bm{x}_{\alpha}, \bm{y}_{\alpha})}\big) \in \C[\bm{x}_{\alpha}, \bm{y}_{\alpha}]^{\mathfrak{S}_{\alpha}}_+$.
Since $I_n^{(2,0)}$ is an ideal, $P \equiv \widetilde{P} \pmod{I_n^{(2,0)}}$; hence
\[ A = \delta_\gamma (P+m) \equiv \delta_\gamma (\widetilde{P} +m) \pmod {I_n^{(2,0)}}.\]

    Since $\cY_\alpha(\delta_\gamma) = \delta_\gamma$ and $\cC_{\alpha}(\widetilde{P}+m) = \widetilde{P}+m$ then $\cY_\alpha(\delta_\gamma (\widetilde{P}+m)) = \delta_\gamma (\widetilde{P} +m)\in \mathcal{A}_\alpha$. 
    Thus by Lemma~\ref{lem:generalized-Vandermonde}, $\delta_\gamma (\widetilde{P} +m)$ is a $\C$-linear combination of generalized Vandermonde determinants $\delta_{\gamma'}$ (for various $\gamma'$).
    Since $I_n^{(2,0)}$ is an ideal closed under the action of $\mathfrak{S}_n$, then $x_i^r y_i^s \cdot I_n^{(2,0)} \subseteq I_n^{(2,0)}$, and $\cC_{\overline{\beta}}(I_n^{(2,0)}) \subseteq I_n^{(2,0)}$, hence
    \begin{equation*}
        \cC_{\overline{\beta}}(x_i^r y_i^s A) \equiv \cC_{\overline{\beta}}(x_i^r y_i^s \delta_\gamma (\widetilde{P}+m)) \pmod {I_n^{(2,0)}},
    \end{equation*}
    and by linearity it suffices to prove that $\cC_{\overline{\beta}}(x_i^r y_i^s \delta_\gamma) \in I_n^{(2,0)} $ for all generalized Vandermonde determinants $\delta_\gamma$.    
    Order the rows of $\delta_\gamma$ (from Lemma~\ref{lem:generalized-Vandermonde}) as follows:
\[
\delta_\gamma = \det [ x_1^\bullet y_1^\bullet, \dots , x_{n-d}^\bullet y_{n-d}^{\bullet}, x_i^\bullet y_i^\bullet]^T.
\]
Then any $\xi \in \mathfrak{S}_{\overline{\beta}}$ with $\xi(i) = k$ acts on $\delta_\gamma$ by replacing each $x_i,y_i$ with $x_k,y_k$, while keeping the exponents and other variables unchanged. Since $\delta_\gamma \in \C[\bm{x}_\alpha, \bm{y}_\alpha]$ with $\overline{\beta} = \overline{\alpha} \cup \{i\}$, this action depends only on the value of $k \in \overline{\beta}$, so we write $\delta_\gamma^{(k)} := \xi \delta_\gamma$ where $\delta_\gamma^{(i)} = \delta_\gamma$. Hence, 
    \begin{align*}
        \cC_{\overline{\beta}}(x_i^r y_i^s \delta_\gamma) 
        &= \frac{1}{d!} \sum_{\sigma \in \S_{\overline{\alpha}}} \left(
        \sigma(x_i^r y_i^s \delta_\gamma) + \sum_{k \in \overline{\alpha}} (i,k)\sigma(x_i^r y_i^s \delta_\gamma) \right) 
        = \frac{1}{d!} \sum_{\sigma \in \S_{\overline{\alpha}}} \left(
        x_i^r y_i^s \delta_\gamma + \sum_{k \in \overline{\alpha}}  x_k^r y_k^s \delta_\gamma^{(k)} \right) \\
        &= \frac{(d-1)!}{d!} \left(
        x_i^r y_i^s \delta_\gamma^{(i)} + \sum_{k \in \overline{\alpha}}  x_k^r y_k^s \delta_\gamma^{(k)} \right) 
        = \frac{1}{d}  \sum_{k \in \overline{\beta}}  x_k^r y_k^s \delta_\gamma^{(k)} .
    \end{align*}

    Next, perform Laplace expansion of $\delta_\gamma^{(k)}$ along its last row, with signs $\eta_m = (-1)^{n-d+1+m}$ and matrix minors $M_m = \det(x_\ell^{c_{m'}}y_\ell^{e_{m'}})_{\ell \in \beta,\ m' \neq m}$ (each of which has no dependence on $k$), to obtain
    \begin{equation}\label{eq:WTS-in-ideal}
    \begin{aligned}
        \sum_{k \in {\overline{\beta}}} x_k^r y_k^s \delta_\gamma^{(k)} &= \sum_{m=1}^{|{\alpha}|} \eta_m M_m p_{c_m+r, e_m+s}(\bm{x}_{\overline{\beta}},\bm{y}_{\overline{\beta}})\\
        &= \sum_{m=1}^{|{\alpha}|} \eta_m M_m p_{c_m+r, e_m+s}(\bm{x},\bm{y}) - \sum_{m=1}^{|{\alpha}|} \eta_m M_m p_{c_m+r, e_m+s}(\bm{x}_\beta,\bm{y}_\beta)  \\
        &\equiv - \sum_{m=1}^{|{\alpha}|} \eta_m M_m p_{c_m+r, e_m+s}(\bm{x}_\beta,\bm{y}_\beta) \pmod {I_n^{(2,0)}}, 
    \end{aligned}
    \end{equation}
    where the last line follows because $c_m+r + e_m+s \geq r+s \geq 1$ by hypothesis, so each $p_{c_m+r, e_m+s}(\bm{x},\bm{y}) \in I_n^{(2,0)}$.
    Then we write
    \begin{equation}\label{eq:Laplace-expansion}
        \sum_{m=1}^{|{\alpha}|} \eta_m M_m p_{c_m+r, e_m+s}(\bm{x}_\beta,\bm{y}_\beta) = \sum_{\ell \in \beta} x_\ell^{r} y_\ell^{s} \sum_{m=1}^{|{\alpha}|} \eta_m x_\ell^{c_m} y_\ell^{e_m} M_m. 
    \end{equation}
    Finally, for each $\ell \in \beta$, observe that the sum $\sum_{m=1}^{|{\alpha}|} \eta_m x_\ell^{c_m} y_\ell^{e_m} M_m$ is (the Laplace expansion along the bottom row of) a determinant of the matrix obtained from $\delta_\gamma$ by replacing the bottom row with row $\ell$. 
    Then since this matrix has two repeated rows, its determinant is $0$. 
    Hence Equation~\eqref{eq:Laplace-expansion} is equal to $0$, so the right-hand side of Equation~\eqref{eq:WTS-in-ideal} lies in $I_n^{(2,0)}$, and thus so does $\cC_{\overline{\beta}}(x_i^r y_i^s \delta_\gamma)$.
\end{proof}

\begin{lemma}\label{lem:key}
    Let $f \in G_\beta$, let $i \in \overline{\beta}$, and let $r,s \geq 0$  with $r+s \geq 1$. Then
    \begin{equation*}
        \cY_{\beta \cup \{i\}} (\partial_{x_i}^r \partial_{y_i}^s f) = 0.
    \end{equation*}
\end{lemma}

\begin{proof}
Fix $i \in \overline{\beta}$ and write $\alpha = \beta \cup \{i\}$ and $\overline{\alpha}= [n] \setminus \alpha$.
Since the apolar form is nondegenerate, to prove that this element is $0$, it suffices to show that it is orthogonal to every polynomial $g \in \C[\bm{x},\bm{y}]$. Moreover, since $f \in G_\beta$, then $\cC_{\overline{\beta}}(f) = f$ so that $\cC_{\overline{\alpha}}(f) = f$. Thus, we compute 
\begin{align*}
\left\langle \cY_\alpha(\partial_{x_i}^r \partial_{y_i}^s f), g\right\rangle_A 
&= \left\langle  f, {x_i}^r{y_i}^s\cY_\alpha(g)\right\rangle_A 
= \left\langle  \cC_{\overline{\alpha}}(f), {x_i}^r{y_i}^s\cY_\alpha(g)\right\rangle_A\\
&= \left\langle  f, \cC_{\overline{\alpha}}({x_i}^r{y_i}^s\cY_\alpha(g))\right\rangle_A 
= \left\langle  f, {x_i}^r{y_i}^s \cC_{\overline{\alpha}}(\cY_\alpha(g))\right\rangle_A\\
&= \left\langle \cC_{\overline{\beta}}(f), {x_i}^r{y_i}^s \cC_{\overline{\alpha}}(\cY_\alpha(g))\right\rangle_A
= \left\langle f,  \cC_{\overline{\beta}}({x_i}^r{y_i}^s \cC_{\overline{\alpha}}(\cY_\alpha(g)))\right\rangle_A.
\end{align*}
By Lemma~\ref{lem:ideal-membership} the result follows.
\end{proof}
  
We define the following function which enables us to construct harmonic elements in $H_n^{(2,1)}$ from those in $\widetilde{G}_\beta$.

\begin{definition}
For each $0\leq d \leq n-1$, define the maps $\Phi_d: \widetilde{G}_{\beta_d}\to H_n^{(2,1)}$, sending each $f \in  \widetilde{G}_{\beta_d}$ to the element
\begin{equation*}
    \Phi_d(f) := \cY_{[n]}(\theta_{\overline{\beta}_d} f),
\end{equation*}
where $\theta_{\overline{\beta_d}} := \theta_{n-d+1} \cdots \theta_n$ denotes the partial product of the fermionic variables. Set $\Phi := \bigoplus_{d=0}^{n-1} \Phi_d$.
\end{definition}

\begin{remark}The lift $\Phi$ may be compared with the explicit isomorphism of Gorsky--Mellit \cite{GorskyMellit}, which relates hook components of $R_n^{(2,0)}$ to sign components of $R_n^{(2,0)} \otimes R_n^{(0,1)}$; we further pass to the quotient $R_n^{(2,1)}$.
\end{remark} 

We now prove that the lift $\Phi_d(f)$ is harmonic.

\begin{theorem}\label{thm:harmonic-2-1}
    For every $f \in \widetilde{G}_{\beta_d}$, we have that $\Phi_d(f) = \cY_{[n]} (\theta_{\overline{\beta_d}} f) \in H_n^{(2,1)}$. Thus, $\Phi_d$ is well-defined.
\end{theorem}

\begin{proof} Fix $0\leq d \leq n-1$ and set $\beta = \beta_d$. 
    To prove that $\Phi_d(f) \in H_n^{(2,1)}$, from Equation~\eqref{eq:diag-super-coinv}, it suffices to verify that $D_{r,s,\varepsilon}(\Phi_d(f)) = 0$ for all $r+s+\varepsilon \geq 1$, where $\varepsilon \in \{0,1\}$.
    Recall also that by Lemma~\ref{lem:commuting-operators}, $D_{r,s,\varepsilon} \Phi_d(f) = \cY_{[n]}(D_{r,s,\varepsilon}(\theta_{\overline{\beta}} f))$.

    Consider the case where $\varepsilon=0$. 
    Since $\partial_{x_j}$ and $\partial_{y_j}$ commute with multiplication by $\theta_{\overline{\beta}}$ for any $j$, then by Equation~\eqref{eq:diag-coinv}, $f \in H_n^{(2,0)}$ implies $D_{r,s,0}(\theta_{\overline{\beta}} f) = \theta_{\overline{\beta}} D_{r,s,0}( f) = 0$.

    Now consider the case where $\varepsilon=1$.
    Since $\partial_{\theta_\ell}(\theta_{{\overline{\beta}}} f)=0$ for all $\ell \in \beta$, 
    then setting $\theta_{{\overline{\beta}} \setminus \{i\}}= \theta_{n-d+1} \dots \theta_{i-1} \theta_{i+1} \dots \theta_n$,
    we can write
    \begin{equation*}
        D_{r,s,1} (\theta_{\overline{\beta}} f) = 
        \sum_{i \in {\overline{\beta}}} (-1)^{i-n+d-1} \theta_{{\overline{\beta}} \setminus \{i\}} \partial_{x_i}^r \partial_{y_i}^s f.
    \end{equation*}
    Fix $i \in {\overline{\beta}}$, let $\alpha := \beta \cup \{i\}$ and $\overline{\alpha} := {\overline{\beta}} \setminus \{i\}$, and consider the coset factorization of $\cY_{[n]}$,
    \[\cY_{[n]} = \frac{(n-d+1)!(d-1)!}{n!}\left(\sum_{\tau \in  \S_n /  (\S_\alpha \times \S_{\overline{\alpha}})} \sgn(\tau) \tau \right) \cY_\alpha \cY_{\overline{\alpha}}
    \]
    over left coset representatives $\tau$ of $\mathfrak{S}_\alpha \times \mathfrak{S}_{\overline{\alpha}}$ in $\mathfrak{S}_n$. 
    Since the monomial $\theta_{\overline{\alpha}} $ is $\mathfrak{S}_{\overline{\alpha}}$-antisymmetric and commutes with the action of $\mathfrak{S}_\alpha$, then by Lemma~\ref{lem:key} and Lemma~\ref{lem:Y_to_C} for any $r+s \geq 1$ we have, 
    \[
\cY_\alpha \cY_{\overline{\alpha}} (\theta_{\overline{\alpha}} \partial_{x_i}^r \partial_{y_i}^s f) 
       =\theta_{\overline{\alpha}} \cY_\alpha\cC_{\overline{\alpha}}(\partial_{x_i}^r \partial_{y_i}^s f)
        = \theta_{\overline{\alpha}} \cC_{\overline{\alpha}}\cY_\alpha(\partial_{x_i}^r \partial_{y_i}^s f)=0.
    \]    
Consider when $r=s=0$. Since $f \in \widetilde{G}_\beta$ gives $\cY_\alpha(f) = 0$, then $\cY_\alpha\cY_{\overline{\alpha}}(\theta_{\overline{\alpha}}f)=\theta_{\overline{\alpha}}\cC_{\overline{\alpha}}\cY_\alpha(f) = 0$. 
    
    Writing these two cases together 
    and summing over all choices of $i \in \overline{\beta}$, we conclude that for any $r,s\geq 0$,
    \begin{align*}
        D_{r,s,1} \Phi_d
        (f) &= \frac{(n-d+1)!(d-1)!}{n!}\sum_{i\in {\overline{\beta}}}
        (-1)^{i-n+d-1}
        \left( \sum_\tau \sgn(\tau) \tau \right) \theta_{\overline{\alpha}} \cC_{\overline{\alpha}} \cY_\alpha (\partial_{x_i}^r \partial_{y_i}^s f) = 0.
    \end{align*}
\end{proof}

\begin{lemma}\label{lem:injectivity}
    The function $\Phi_d$ is an injective, $(\bm{x},\bm{y})$-bidegree preserving, linear map such that for each $f \in  \widetilde{G}_{\beta_d}$, the image $\Phi_d(f)$ is homogeneous of $\bm{\theta}$-degree $d$. Thus, $\Phi_d(\widetilde{G}_{\beta_d})$ is contained in the $\bm{\theta}$-degree $d$ component of $H_n^{(2,1)}$, 
    \[
\Phi_d: \widetilde{G}_{\beta_d} \hookrightarrow  \left(H_n^{(2,1)}\right)_{\deg_{\bm{\theta}} = d}.
    \]
\end{lemma}

\begin{proof}
As above, fix $0\leq d \leq n-1$ and set $\beta = \beta_d$. 
First observe that, for each $\sigma \in \S_n$, a term $\sigma(\theta_{\overline{\beta}} f)=\sigma(\theta_{\overline{\beta}})\sigma(f)$ contributes to the coefficient of $\theta_{\overline{\beta}}$
in $\Phi_d(f) = \cY_{[n]} (\theta_{\overline{\beta}} f)$
only when $\sigma(\theta_{\overline{\beta}}) = \pm \theta_{\overline{\beta}}$. Hence, it suffices to restrict our attention only to permutations $\sigma \in \mathfrak{S}_{n-d} \times \mathfrak{S}_{d} \cong \mathfrak{S}_\beta \times \mathfrak{S}_{\overline{\beta}}$.

Now, observe that since $f \in G_\beta$ and $\sigma(\theta_{\overline{\beta}} )= \sgn(\sigma)\theta_{\overline{\beta}} $ for all $\sigma \in \S_{\overline{\beta}}$, then by Lemma~\ref{lem:Y_to_C}:
\begin{align*}
\cY_\beta \cY_{\overline{\beta}} (\theta_{\overline{\beta}} f)
= \theta_{\overline{\beta}} \cY_\beta \cC_{\overline{\beta}} (f) 
= \theta_{\overline{\beta}} f.
\end{align*}
Hence, using the coset factorization of $\cY_{[n]}$ over left coset representatives $\tau$ of $\mathfrak{S}_\beta \times \mathfrak{S}_{\overline{\beta}}$ in $\mathfrak{S}_n$, we obtain
  \begin{align*}
  \frac{n!}{(n-d)!d!}\cY_{[n]}(\theta_{\overline{\beta}} f)
  = \left(\sum_{\tau \in  \S_n /(  \S_\beta \times \S_{\overline{\beta}})} \sgn(\tau) \tau \right)  \cY_{\overline{\beta}} \cY_\beta (\theta_{\overline{\beta}} f)
  = \left(\sum_{\tau \in  \S_n /(  \S_\beta \times \S_{\overline{\beta}})} \sgn(\tau) \tau \right)(\theta_{\overline{\beta}} f). 
    \end{align*} 
    Since the only terms in this sum that contribute to  $\theta_{\overline{\beta}}$ occur when $\tau$ is the identity, then it follows that the coefficient of $\theta_{\overline{\beta}}$ in $\cY_{[n]}(\theta_{\overline{\beta}} f)$ is precisely
    \begin{align*}
  \frac{(n-d)!d!}{n!} f. 
    \end{align*} 
    Since this coefficient is nonzero, then the linear map sending $f \mapsto \Phi_d(f)$ is injective.   
    Finally, $\Phi_d(f)$ has the same $(\bm{x},\bm{y})$ bidegree as $f$, and is homogeneous of $\bm{\theta}$-degree $|{\overline{\beta}}| =d$ by construction.
\end{proof}

\subsection{The main theorem}\label{subsec:main theorem}

For an $\mathfrak{S}_n$-module $M$, write $M_{\sgn}$ for its sign-isotypic component, so that $M_{\sgn} \cong \epsilon \otimes \Hom_{\mathfrak{S}_n} (\epsilon, M)$ with graded dimension equal to $\dim(\Hom_{\mathfrak{S}_n} (\epsilon, M))$. 
We are now able to prove the first part of our main result.

\begin{theorem}\label{thm:Phi-injection}
For any $n\geq 1$, $\Phi$ induces a grading-preserving, injective, linear map
\[
\Hom_{\mathfrak{S}_n}\left(\epsilon, H_n^{(2,0)} \otimes H_n^{(0,1)}\right) \hookrightarrow \Hom_{\S_n}\left(\epsilon, H_n^{(2,1)}\right).
\]
\end{theorem}

\begin{proof}
    Fix $n\geq 1$ and $0 \leq d \leq n-1$.
    Since $\sigma \cY_{[n]} = \sgn(\sigma) \cY_{[n]}$ for all $\sigma \in \S_n$, for any $f \in \widetilde{G}_{\beta_d}$ the element $\Phi_d(f)$ is alternating and thus contained in the sign-isotypic component of $H_n^{(2,1)}$.
    So by Lemma~\ref{lem:injectivity}, 
    we obtain an injective map of bigraded vector spaces
    \[
\Phi_d: \widetilde{G}_{\beta_d} \hookrightarrow \epsilon \otimes \Hom_{\S_n}\left(\epsilon,\left(H_n^{(2,1)}\right)_{\deg_{\bm\theta }=d}\right) \cong \Hom_{\S_n}\left(\epsilon,\left(H_n^{(2,1)}\right)_{\deg_{\bm{\theta}} =d}\right).
    \]
By Theorem~\ref{thm:tildehook-hom-iso} there is also an isomorphism
    \[
    \widetilde{G}_{\beta_d} \cong \Hom_{\S_n}(S^{(d+1,1^{n-d-1})}, H_n^{(2,0)}).
    \]
Summing over all $\bm{\theta}$-degree $d$ components, and via Equation~\eqref{eqn:hook-sign-iso}, the sum $\Phi =\bigoplus_d \Phi_d$ induces a grading-preserving injective map $\widetilde{\Phi}:\Hom_{\S_n}\left( \epsilon, H_n^{(2,0)} \otimes H_n^{(0,1)}\right) \hookrightarrow \Hom_{\S_n}\left(\epsilon, H_n^{(2,1)}\right)$ of triply-graded vector spaces:
\[
\begin{tikzcd}
\Hom_{\S_n}\left( \epsilon, H_n^{(2,0)} \otimes H_n^{(0,1)}\right) \arrow[d,hook,"\widetilde{\Phi}"{font=\normalsize}, start anchor = {[xshift=-30ex]}, end anchor = {[xshift=-20ex]}]
\cong \displaystyle\bigoplus_{d=0}^{n-1}\Hom_{\S_n}\left( S^{(d+1,1^{n-d-1})}, H_n^{(2,0)}\right) 
\cong \bigoplus_{d=0}^{n-1} \widetilde{G}_{\beta_d}
\arrow[d,hook',"\Phi"' {font=\large},start anchor = {[xshift=30ex]}, end anchor = {[xshift=20ex]}]\\
\Hom_{\S_n}\left(\epsilon, H_n^{(2,1)}\right) =\displaystyle\bigoplus_{d=0}^{n-1} \Hom_{\S_n}\left(\epsilon,\left(H_n^{(2,1)}\right)_{\deg_{\bm{\theta}} =d}\right).
\end{tikzcd}
\]
\end{proof}

By Theorem \ref{thm:Phi-injection}, the graded dimension of $(H_n^{(2,0)} \otimes H_n^{(0,1)})_{\sgn}$ is bounded above by that of $(H_n^{(2,1)})_{\sgn}$. 
Combined with Lemma~\ref{lem:quotient}, we obtain the promised isomorphism.

\begin{theorem}\label{thm:main}
For $n \geq 1$, the quotient map $\rho: R_n^{(2,0)} \otimes R_n^{(0,1)} \twoheadrightarrow R_n^{(2,1)}$ restricts to an isomorphism on the sign-isotypic components.
\end{theorem}

\begin{proof}
Denote by $\eta_{(k,\ell)}: H_n^{(k,\ell)} \to R_n^{(k,\ell)}$ the isomorphisms of multigraded $\S_n$-modules considered in Proposition \ref{prop:H-R-isomorphism} for $(k,\ell)\in\{(2,0),(2,1),(0,1)\}$. Then, by Theorem \ref{thm:Phi-injection}, we have an injective map of triply-graded vector spaces from $ \Hom_{\S_n}\left(\epsilon, R_n^{(2,0)} \otimes R_n^{(0,1)}\right)$ into  $\Hom_{\S_n}\left(\epsilon, R_n^{(2,1)}\right)$, given by the composition:
\begin{equation}\label{eq:injection on R}
\begin{tikzcd}
    \Hom_{\S_n}\left(\epsilon, R_n^{(2,0)} \otimes R_n^{(0,1)}\right)\arrow[d, "(\eta_{(2,0)}^{-1}\otimes \eta_{(0,1)}^{-1})\circ(-) "]\arrow[r,hook]
    &
    \Hom_{\S_n}\left(\epsilon, R_n^{(2,1)}\right)
    \\
     \Hom_{\S_n}\left(\epsilon, H_n^{(2,0)} \otimes H_n^{(0,1)}\right)\arrow[r,hook, "\widetilde{\Phi}"]
    &
    \Hom_{\S_n}\left(\epsilon, H_n^{(2,1)}\right)\arrow[u, "\eta_{(2,1)}\circ (-)"].
\end{tikzcd}
\end{equation}

Recall from Lemma~\ref{lem:quotient} that the quotient map $\rho : R_n^{(2,0)} \otimes R_n^{(0,1)} \twoheadrightarrow  R_n^{(2,1)}$ is a surjection of triply-graded $\mathfrak{S}_n$-modules.
By Schur's lemma, the map restricts to a surjection on the sign-isotypic components
\begin{equation}\label{eq:sign-surjection}
\rho|_{\sgn} : (R_n^{(2,0)} \otimes R_n^{(0,1)})_{\sgn}\twoheadrightarrow  (R_n^{(2,1)})_{\sgn},
\end{equation}
so we have an induced grading-preserving surjection on the multiplicity spaces
\begin{equation}\label{eq:hom-surjection}
  \Hom_{\S_n}(\epsilon, R_n^{(2,0)} \otimes R_n^{(0,1)}) \twoheadrightarrow
    \Hom_{\S_n}(\epsilon, R_n^{(2,1)}).
    \end{equation}
Combined with Equation \eqref{eq:injection on R}, since both spaces are finite-dimensional in each tridegree, Equation~\eqref{eq:hom-surjection} is an isomorphism, hence Equation~\eqref{eq:sign-surjection} is as well.
\end{proof}

\begin{remark}
 In Proposition \ref{prop:H-R-isomorphism} there is an isomorphism between $H_n^{(k,\ell)}$ and $R_n^{(k,\ell)}$ as graded $\S_n$-modules. 
 These spaces admit a different identification as graded duals.
 Using this, we obtain the following alternative proof of Theorem \ref{thm:main} communicated to us by Eugene Gorsky.
 
 Rota's apolar form in Equation \eqref{eq:Rota-apolar-form} induces a perfect degree-preserving bilinear pairing $H_n^{(2,0)} \times R_n^{(2,0)} \to \C$, so that $R_n^{(2,0)} \cong (H_n^{(2,0)})^*$ as bigraded $\S_n$-modules. By extending the definition of Rota's apolar form to include fermionic variables (see \cite{SwansonWallach1} and \cite[Appendix A]{JiangLentfer}), one obtains similar identifications $R_n^{(0,1)} \cong (H_n^{(0,1)})^*$ and $R_n^{(2,1)} \cong (H_n^{(2,1)})^*$. 
Since all these spaces are finite-dimensional, the injection in Theorem \ref{thm:Phi-injection} induces a surjective map on their duals, which by the identification above yields a surjection
\[
\Hom_{\S_n}\left(\epsilon, R_n^{(2,1)}\right)
\twoheadrightarrow
\Hom_{\mathfrak{S}_n}\left(\epsilon, R_n^{(2,0)} \otimes R_n^{(0,1)}\right).
\]
Combined with Equation \eqref{eq:hom-surjection} we again obtain the desired isomorphism.
\end{remark}

As a consequence of Theorem \ref{thm:main}, we obtain the following combinatorial formula for the Hilbert series of $(R_n^{(2,1)})_{\sgn}$.

\begin{corollary}\label{cor:sign-character-schroder}
    For all $n\geq 1$,
\begin{equation*}
    \Hilb((R_n^{(2,1)})_{\sgn}; q,t;a) = \left\langle \Frob(R_n^{(2,1)}; q,t;a), e_n\right\rangle =  \widetilde{S}_n(q,t,a).
\end{equation*}
\end{corollary}

\begin{proof}
    Taking the Frobenius character of both sides of the isomorphism in~\eqref{eq:sign-surjection} and applying Corollary~\ref{cor:(2,0)(0,1)bound} we obtain the equalities
\[
    \left\langle \Frob(R_n^{(2,1)}; q,t;a), e_n\right\rangle
    = \left\langle \Frob(R_n^{(2,0)} \otimes R_n^{(0,1)}; q,t;a), e_n\right\rangle
    = \widetilde{S}_n(q,t,a).
    \]
\end{proof}

By specializing $q=t=a=1$ in Corollary~\ref{cor:sign-character-schroder} we prove the following enumerative conjecture of F. Bergeron \cite[Table 3]{Bergeron2020}.
\begin{corollary}\label{cor:Bergeron-conj}
    For $n \geq 1$, the sign character of $R_n^{(2,1)}$ has multiplicity $\widetilde{S}_n$.
\end{corollary}

\subsection{The sign character in Zabrocki's conjecture} \label{subsec:zabrocki-conj}
The Delta conjecture of Haglund, Remmel, and Wilson \cite{HaglundRemmelWilson2018}, now a theorem, gives a combinatorial formula in terms of decorated Dyck paths for expressions of the form $\Delta'_{e_k}(e_n)$. 
Its sign-character case was proven by Zabrocki as a part of the $4$-variable Catalan theorem \cite{Zabrocki-4Catalan-2016}, and the full Delta conjecture was proven first by D'Adderio and Mellit \cite{DAdderioMellit} and also by Blasiak, Haiman, Morse, Pun, and Seelinger \cite{BHMPS-Delta}. 
D'Adderio, Iraci, and Vanden Wyngaerd \cite[Section~6]{D_Adderio_2021} observed that the sign-character case is the Schr\"oder case of their compositional Delta conjecture, also proven in \cite{DAdderioMellit}.

Zabrocki \cite{Zabrocki2019} conjectured (see Conjecture~\ref{conj:zabrocki}) that the diagonal superspace coinvariant ring $R_n^{(2,1)}$ is a module for the Delta theorem, that is,
\begin{equation}\label{eq:zabrocki}
    \Frob(R_n^{(2,1)};q,t;a) = \sum_{d=0}^{n-1} a^d \Delta'_{e_{n-1-d}}(e_n).
\end{equation}
The conjecture was verified by computer for $n \leq 6$ in \cite{Zabrocki2019} and remains open. 
In what follows we prove it for the sign-isotypic component.

We begin by relating the right hand side of \eqref{eq:zabrocki} to the little Schr\"oder polynomials. Although Proposition~\ref{prop:zabrocki-alternating} follows from Zabrocki's 4-variable Catalan theorem, we provide an elementary proof below.

\begin{proposition}\label{prop:zabrocki-alternating}
For $n \geq 1$, we have the equality
    \begin{equation*}
        \left\langle \sum_{d=0}^{n-1} a^d\Delta'_{e_{n-1-d}}(e_n),e_n
        \right\rangle = \widetilde{S}_n(q,t,a).
    \end{equation*}
\end{proposition}

\begin{proof}
Extracting coefficients of $a^d$, by linearity of the Hall inner product it suffices to prove that 
 $ \widetilde{S}_{n,d}(q,t) = \left\langle \Delta'_{e_{n-1-d}}(e_n),e_n \right\rangle $ for each $0 \leq d \leq n-1$.
Combining \cite[Corollary 4.14.2]{Haglund2008} and \cite[Corollary 4.16.1]{Haglund2008}, we deduce that
\begin{equation*}
S_{n,d}(q,t) = \left\langle \Delta_{h_n}(e_{n+1-d}), h_{n+1-d} \right\rangle = \left\langle \Delta_{e_{n-d}}(e_n),e_n \right\rangle.
\end{equation*}
Thus, by Proposition~\ref{prop:Schroder-alternating} and the fact that $\Delta_{e_{n-d}}(e_n) = \Delta'_{e_{n-d}}(e_n) + \Delta'_{e_{n-d-1}}(e_n)$, we obtain that
\begin{align*}
    \widetilde{S}_{n,d}(q,t) 
    &= \sum_{i=0}^d (-1)^i S_{n,d-i}(q,t) 
    \\
    &= \sum_{i=0}^d (-1)^i \left\langle \Delta'_{e_{n-d+i}}(e_n),e_n \right\rangle  + \sum_{i=0}^d (-1)^i \left\langle \Delta'_{e_{n-d+i-1}}(e_n),e_n \right\rangle
    \\
    &= (-1)^d \left\langle \Delta'_{e_{n}}(e_n),e_n \right\rangle  + \sum_{i=0}^{d-1} ((-1)^i+(-1)^{i+1}) \left\langle \Delta'_{e_{n-d+i}}(e_n),e_n \right\rangle 
    + \left\langle \Delta'_{e_{n-d-1}}(e_n),e_n \right\rangle \\
    &= \left\langle \Delta'_{e_{n-d-1}}(e_n),e_n \right\rangle 
\end{align*}
since $\Delta'_{e_{n}}(e_n)=0$.
\end{proof}

Combining Proposition~\ref{prop:zabrocki-alternating} with Theorem~\ref{thm:main}, we obtain the sign-character component of Zabrocki's conjecture. 

\begin{corollary}\label{cor:zabrocki-conjecture-sign} 
For $n \geq 1$, the sign-character component of Conjecture~\ref{conj:zabrocki} is true:
    \begin{equation*}
    \left\langle \Frob(R_n^{(2,1)};q,t;a), e_n\right\rangle = \left\langle \sum_{d=0}^{n-1} a^d \Delta'_{e_{n-1-d}}(e_n), e_n \right\rangle.
\end{equation*}
\end{corollary}

\subsection{On the \texorpdfstring{$a=-t$}{a=-t} specialization}

Recall the \newword{classical coinvariant ring} $R_n^{(1,0)} = \C[\bm{x}]/I_n^{(1,0)}$, where $I_n^{(1,0)} = \left\langle \C[\bm{x}]^{\mathfrak{S}_n}_+ \right\rangle$. 
By \cite{Lentfer-Supersymmetry}, we have that
\begin{equation*}
    \Frob(R_n^{(2,1)};q,t;a)|_{a=-t} = \Frob(R_n^{(1,0)};q).
\end{equation*}
Since the sign-isotypic component of $R_n^{(1,0)}$ is spanned by the Vandermonde determinant $\delta_n$, which is homogeneous of degree $\binom{n}{2}$, we have
\begin{equation*}
    \left\langle \Frob(R_n^{(1,0)};q), e_n \right\rangle = q^{\binom{n}{2}}.
\end{equation*}
Combined with Corollary~\ref{cor:sign-character-schroder}, this yields the specialization
\begin{equation*}\label{eq:a-equals-minus-t}
    \sum_{d=0}^{n-1} (-t)^d \widetilde{S}_{n,d}(q,t) = q^{\binom{n}{2}}.
\end{equation*}
At $q=t=1$, this yields the enumerative identity
$\sum_{d=0}^{n-1}(-1)^d |\tilde{L}_{n,n,d}^+| = 1$.    
While these identities look simple, they appear difficult to prove combinatorially, because a sign-reversing involution would have to interchange the number of diagonal steps with the bounce statistic (equivalently, after applying the zeta map, with the area statistic).

\section{Connections with link homology}\label{sec:link-homology}
\newword{Khovanov--Rozansky homology} is a triply-graded homology theory which categorifies the HOMFLY-PT polynomial \cite{KR1,KR2,Kh}. In recent years, a plethora of connections (many still conjectural) to $(q,t)$-Catalan and Schr\"oder combinatorics have arisen \cite{Gorsky-Catalan,Hog17, GL1, GL2,GM1,GM2,GMV1,GMV2,Mellit,CGHM}. 
Fundamentally, these results assert that the Poincar\'e series $\mathcal{P}_K(q,t,a)$ of the Khovanov--Rozansky homology of certain knots $K$ coincides with the sum of the hook components of the symmetric functions coming from the various Shuffle theorems \cite{CarlssonMellit2018, Mellit, BHMPS-Paths}. 
In the most classical setting when $K = T(n,n+1)$ the $(n,n+1)$-torus knot, this correspondence, conjectured by Gorsky \cite{Gorsky-Catalan} and proven by Hogancamp \cite{Hog17}, states:

\begin{theorem}[\cite{Hog17}]\label{thm:Hog-torus}
The Poincar\'e series $\mathcal{P}_K(q,t,a)$ of the triply-graded Khovanov--Rozansky homology of the $(n, n+ 1)$-torus knot $K$ equals
\begin{equation} \label{eq:torus homology}  \mathcal{P}_K(q,t,a)=\frac{1}{1-q}\sum_{d=0}^n \Big\langle \nabla e_n, h_de_{n-d} \Big\rangle a^d.
\end{equation}
\end{theorem}

Combining this with Equation~\eqref{eq:Schroder-polynomial-def-second} and Proposition~\ref{prop:BigSmall-Schroder}, we obtain, for $K = T(n,n+1)$,
\begin{equation*}
\mathcal{P}_K(q,t,a) 
= \frac{1}{1-q} 
S_{n}(q,t,a) = \frac{1+a}{1-q} \; 
\widetilde{S}_{n}(q,t,a).
\end{equation*}
Consequently, putting this together with Corollary \ref{cor:sign-character-schroder}, we obtain the following. 

\begin{theorem}\label{thm:torusknot} 
The sign component of the Frobenius character of the diagonal superspace coinvariant ring $R_n^{(2,1)}$ computes the Poincar\'e series of the Khovanov--Rozansky homology of the $(n,n+1)$-torus knot $K$. Namely, 
  \begin{equation*}\label{eq:KR-(2,1)}
  \mathcal{P}_K(q,t,a) = \frac{1+a}{1-q} \; 
\left\langle \Frob(R_n^{(2,1)}; q,t;a), e_n \right\rangle . 
\end{equation*}
\end{theorem}

\subsection{Comparison with Gorsky--Mellit} 
In light of Theorem~\ref{thm:Hog-torus}, it is evident from Corollary~\ref{cor:(2,0)(0,1)bound} that for the $(n,n+1)$-torus knot $K$:
\begin{equation} \label{eq:KR-tensor}
\mathcal{P}_K(q,t,a) = \frac{1+a}{1-q} \; 
\left\langle \Frob(R_n^{(2,0)} \otimes R_n^{(0,1)}; q,t;a), e_n \right\rangle. 
\end{equation}
In a very recent paper \cite{GorskyMellit}, Gorsky and Mellit gave a module-theoretic realization of \eqref{eq:KR-tensor} by constructing an explicit isomorphism between the Khovanov--Rozansky homology of $T(n,n+1)$, the hook components of $R_n^{(2,0)}$, and the sign component of $R_n^{(2,0)} \otimes R_n^{(0,1)}$ (see \cite[Section 3.3]{GorskyMellit}). 

Despite the similarities, it is important to note that their result does not imply our Theorem~\ref{thm:torusknot}. 
Indeed, Theorem~\ref{thm:main} establishes that the quotient map $R_n^{(2,0)} \otimes R_n^{(0,1)} \twoheadrightarrow R_n^{(2,1)}$ restricts to an isomorphism on sign-isotypic components; Theorem~\ref{thm:torusknot} then shows the same homology is computed by the diagonal superspace coinvariant ring $R_n^{(2,1)}$ itself.

\subsection{Beyond the \texorpdfstring{$(n,n+1)$}{(n,n+1)}-torus knot case}
In \cite{hog-mellit}, Hogancamp and Mellit extended Theorem~\ref{thm:Hog-torus} to all $(n,m)$-torus knots (in fact links), $T(n,m)$. 
More generally, the \emph{$(q,t)$-Schr\"oder theorem for paths under any line} \cite[Section 6]{CGHM} asserts that for any triangular partition $\lambda$, the Khovanov--Rozansky homology of the Coxeter knot $K_\lambda$ can be computed from the hook components of the symmetric function $\mathcal{H}_\lambda(X;q,t)$ arising in the Shuffle theorem \cite{BHMPS-Paths} for the path with maximal partition $\lambda$. 
Specifically, in \cite[Theorem 5.19 and Corollary 6.12]{CGHM}, Caprau, the first author, Hogancamp, and Mazin prove:\footnote{The difference in the coefficients between Equations \eqref{eq:torus homology} and \eqref{eq:coxeter homology} arises from different normalizations for the Poincar\'e polynomials being used in \cite{Hog17} versus \cite{CGHM}.}
\begin{equation}\label{eq:coxeter homology}
  \mathcal{P}_{K_\lambda}(q,t,a)=
  \frac{(at^{-1/2}q^{-1/2})^{|\lambda|}}{1-q}
  S_\lambda(q,t,a),
\end{equation}
where $S_\lambda(q,t,a)$ is the so-called \emph{triangular Schr\"oder polynomial}, which computes the hook components of $\mathcal{H}_\lambda(X;q,t)$ and recovers $S_n(q,t,a)$ when $\lambda=(n-1,n-2,\dots,1)$, the staircase partition. 

For an arbitrary partition $\lambda$, the representation theoretic interpretation of $\mathcal{H}_\lambda(X;q,t)$ as the Frobenius character of some symmetric group module is currently unknown.
However, when $\lambda$ is the maximal partition under the line $x/m+y/n=1$ with $m,n$ relatively prime, so that $K_\lambda$ is the $(n,m)$-torus knot, $T(n,m)$, the corresponding module arises as the associated graded of a certain filtration of the unique irreducible finite dimensional representation $L_{\frac{m}{n}}$ of the rational Cherednik algebra $H_{\frac{m}{n}}$ \cite{Xinchun, GorskyNegut, Gorsky_Oblomkov_Rasmussen_Shende_2014, Hikita}. 
In light of Theorem~\ref{thm:torusknot} and the work of Gorsky--Mellit \cite{GorskyMellit} it would be very interesting to search for a super version of $L_{\frac{m}{n}}$ whose sign-isotypic component recovers the Khovanov--Rozansky homology of $T(n,m)$.

\section*{Computational tools and AI usage}

Computations were performed in SageMath and Macaulay2, utilizing \cite{Zabrocki-code, Bergeron-code}. 
AI tools were used for coding in support of the proof of Lemma~\ref{lem:ideal-membership}, and proofreading.
The mathematics and writing of the paper is human-generated, and the authors take full responsibility for the results and content of the paper.

\bibliographystyle{amsplain}
\bibliography{biblio}
\end{document}